\documentclass[11pt,a4paper]{scrartcl}
\usepackage[utf8]{inputenc}
\usepackage{amsmath}
\usepackage{amsfonts}
\usepackage{amssymb}
\usepackage{amsthm}
\usepackage{graphicx}
\usepackage{enumitem}
\usepackage{lmodern}
\usepackage{microtype}
\usepackage[]{scrlayer-scrpage}
\usepackage[english]{babel}
\usepackage[table]{xcolor}
\usepackage{tabularx} 
\usepackage{xspace}		
\usepackage{setspace}   
\usepackage[autostyle]{csquotes} 
\usepackage{xpatch} 
\usepackage{aliascnt} 
\usepackage{tikz}
\usepackage{pgfplots}

\pgfplotsset{compat=1.14}

\usepackage[
style=numeric,    		
isbn=false,                
url=false,
doi=false,
pagetracker=true,          
maxbibnames=5,            
maxcitenames=3,            
giveninits=true,			
autocite=inline,           
block=space,               
date=comp,                
dateabbrev=false,			
backend=biber,
]{biblatex}
\DeclareFieldFormat[article,incollection,inproceedings,thesis]{citetitle}{#1\midsentence} 
\DeclareFieldFormat[article,incollection,inproceedings,thesis]{title}{#1\midsentence}
\DeclareFieldFormat{titlecase}{\MakeSentenceCase*{#1}} 
\DeclareFieldFormat{titlecase}{\MakeTitleCase{#1}}

\newrobustcmd{\MakeTitleCase}[1]{%
	\ifthenelse{\ifcurrentfield{booktitle}\OR\ifcurrentfield{booksubtitle}%
		\OR\ifcurrentfield{maintitle}\OR\ifcurrentfield{mainsubtitle}%
		\OR\ifcurrentfield{journaltitle}\OR\ifcurrentfield{journalsubtitle}%
		\OR\ifcurrentfield{issuetitle}\OR\ifcurrentfield{issuesubtitle}%
		\OR\ifentrytype{book}\OR\ifentrytype{mvbook}\OR\ifentrytype{bookinbook}%
		\OR\ifentrytype{booklet}\OR\ifentrytype{suppbook}%
		\OR\ifentrytype{collection}\OR\ifentrytype{mvcollection}%
		\OR\ifentrytype{suppcollection}\OR\ifentrytype{manual}%
		\OR\ifentrytype{periodical}\OR\ifentrytype{suppperiodical}%
		\OR\ifentrytype{proceedings}\OR\ifentrytype{mvproceedings}%
		\OR\ifentrytype{reference}\OR\ifentrytype{mvreference}%
		\OR\ifentrytype{report}\OR\ifentrytype{thesis}}
	{#1}
	{\MakeSentenceCase{#1}}}
\renewrobustcmd*{\bibinitdelim}{\,} 
\xpatchbibdriver{online}{\usebibmacro{date}}{}{}{}
\xpatchbibdriver{online}{ \usebibmacro{addendum+pubstate}}{\addcomma\space\usebibmacro{publisher+location+date}\newunit\newblock\usebibmacro{addendum+pubstate}}{}{} 
\bibliography{../../Draft/bib.bib}  
\usepackage{url} 
\usepackage[hidelinks]{hyperref} 
\usepackage[capitalise,nameinlink,noabbrev]{cleveref}
\crefformat{equation}{#2(#1)#3}

\theoremstyle{definition}
\newtheorem{definition}{Definition}

\theoremstyle{plain}
\newtheorem{theorem}{Theorem}[section]
\newaliascnt{lemma}{theorem}
\newtheorem{lemma}[lemma]{Lemma}
\aliascntresetthe{lemma}
\newaliascnt{corollary}{theorem}
\newtheorem{corollary}[corollary]{Corollary}
\aliascntresetthe{corollary}
\newaliascnt{proposition}{theorem}
\newtheorem{proposition}[proposition]{Proposition}
\aliascntresetthe{proposition}

\newtheorem*{thm:ZP}{Theorem \ref{th:ZP_inequivalence1}}
\newtheorem*{cor:ZP}{Corollary \ref{cor:numberofAPN1}}

\theoremstyle{remark}

\addto\extrasenglish{%
}

\newcommand{\F}{\mathbb{F}}

\newcommand{\Fpm}{\mathbb{F}_{2^m}}
\newcommand{\Fptwom}{\mathbb{F}_{2^{2m}}}

\newcommand{\Aut}{\textnormal{Aut}}

\newcommand{\GamL}{\textnormal{$\Gamma$L}}

\begin{document}
	\title{A lower bound on the number of inequivalent APN functions}
	\author{Christian Kaspers\thanks{Institute for Algebra and Geometry, Otto von Guericke University Magdeburg, 39106 Magdeburg, Germany (email: \href{mailto:christian.kaspers@ovgu.de}{\nolinkurl{christian.kaspers@ovgu.de}})} \space and Yue Zhou\thanks{Department of Mathematics, National University of Defense Technology,  410073 Changsha, China (email: \href{mailto:yue.zhou.ovgu@gmail.com}{\nolinkurl{yue.zhou.ovgu@gmail.com}})}}
	\maketitle
	
	\begin{abstract}
		In this paper, we establish a lower bound on the total number of inequivalent APN functions on the finite field with $2^{2m}$ elements, where $m$ is even. We obtain this result by proving that the APN functions introduced by Pott and the second author~\cite{zhou2013}, that depend on three parameters $k$, $s$ and $\alpha$, are pairwise inequivalent for distinct choices of the parameters $k$ and $s$. Moreover, we determine the automorphism group of these APN functions.
	\end{abstract}
	

\section{Introduction}
\label{sec:Introduction}
	Denote by $\F_{2}^n$ the $n$-dimensional vector space over the finite field $\F_2$ with two elements. A function from $\F_2^n$ to $\F_2^m$ is called a \emph{vectorial Boolean function} if $m \ge 2$ or simply a \emph{Boolean function} if $m=1$. Vectorial Boolean and Boolean functions are of particular interest in cryptography but they also have important applications in coding theory and design theory. In this paper, we consider vectorial Boolean functions from $\F_{2}^n$ to $\F_2^n$, we say, functions \emph{on} $\F_2^n$, with optimal differential properties: they are called \emph{almost perfect nonlinear functions}, in brief \emph{APN functions}, and, from a cryptographic standpoint, they offer the best resistance to the differential attack.\par
	
	APN functions have been studied for several decades. While first, only power APN~functions $x \mapsto x^d$ were known, by now, numerous infinite families of non-power APN functions have been found. In \autoref{sec:known_classes}, we present a short overview over the known APN functions. The most fascinating open problem regarding APN functions is whether in vector spaces of even dimension more than one APN permutation exists. So far, only one such function is known: it exists on $\F_2^6$ and was found by \textcite{dillon2010}.\par 
	
	Besides this big APN problem, there are several other intriguing open questions which are related to APN functions: it is, for example, unknown how many APN functions exist on $\F_2^n$ for any given $n$. While there are, as mentioned above, several infinite families of APN functions, many of these constructions provide equivalent functions. By equivalent we mean that there is some transformation that defines an equivalence relation between vectorial Boolean functions and preserves the APN property. So, to reformulate the problem more precisely: it is unknown how many \emph{inequivalent} APN functions exist on the vector space $\F_2^n$.\par 
	
	As far as power APN functions are concerned, equivalence problems are relatively well studied. It has also been shown that many of the non-power APN functions are inequivalent to the power functions. For the several classes of non-power APN functions, however, it is in many cases neither clear how many inequivalent functions a construction provides nor whether two distinct constructions always lead to inequivalent functions.\par
	
	In this paper, we present a first benchmark on the number of inequivalent APN~functions on the $2m$-dimensional vector space $\F_{2}^{2m}$, where $m$ is even. We establish a lower bound on this number by proving that the non-power APN functions found by Pott and the second author~\cite{zhou2013}, which depend on three parameters, are inequivalent for different choices of two of the parameters. We state our main results here already. Their proofs can be found in \autoref{sec:ZP_APN}. Note that, in \autoref{th:ZP_inequivalence1}, we identify the vector space $\F_2^{2m}$ with the finite field $\Fptwom$, and we describe the Pott-Zhou APN~function by its bivariate representation which will be explained in \autoref{sec:Preliminaries}.
	
	\begin{theorem}
		\label{th:ZP_inequivalence1}
			Let $m\ \ge 4$ be an even integer. Let $k,\ell$ be integers coprime to $m$ such that $0 < k,\ell < \frac{m}{2}$, let $s,t$ be even integers with $0 \le s,t \le \frac{m}{2}$, and let $\alpha,\beta \in \Fpm^*$ be non-cubes. Two Pott-Zhou APN~functions $f_{k,s,\alpha}, f_{\ell,t,\beta} \colon \Fptwom \to \Fptwom$, where
			\begin{align*}
				f_{k,s,\alpha}(x,y) = \left(x^{2^k+1} + \alpha y^{(2^k+1)2^s},\ xy\right)	&&\text{and}	&& f_{\ell,t,\beta}(x,y) = \left(x^{2^\ell+1} + \beta y^{(2^\ell+1)2^t},\ xy\right),
			\end{align*}
			are CCZ-equivalent if and only if $k = \ell$ and $s = t$.
	\end{theorem}

	With the help of \autoref{th:ZP_inequivalence1}, we immediately obtain a lower bound on the number of inequivalent APN functions on $\F_2^{2m}$, where $m$ is even:
	\begin{corollary}
	\label{cor:numberofAPN1}
		On $\Fptwom$, where $m \ge 4$ is even, there exist 
		\[
			\left(\left\lfloor\frac{m}{4}\right\rfloor+1\right) \frac{\varphi(m)}{2}
		\]
		CCZ-inequivalent Pott-Zhou APN functions, where $\varphi$ denotes Euler's totient function.
	\end{corollary}

\section{Preliminaries}
\label{sec:Preliminaries}
	In this section, we will present all the definitions and basic results needed to follow the paper. From now on, we will only consider vectorial Boolean functions on $\F_2^n$, and we will, in most cases, identify the $n$-dimensional vector space $\F_2^n$ over $\F_{2}$ with the finite field $\F_{2^n}$ with $2^n$ elements. This will allow us to use finite field operations and notations. Note that any function on the finite field $\F_{2^n}$ can be written as a univariate polynomial mapping of degree at most $2^n-1$. Furthermore, denote by $\F_{2^n}^*$ the multiplicative group of $\F_{2^n}$. We start by defining APN functions.
	
	\begin{definition}
	\label{def:APN}
		A function $f \colon \F_{2^n} \to \F_{2^n}$ is called \emph{almost perfect nonlinear} (APN) if the equation 
		\[
			f(x+a) + f(x) = b
		\]
		has exactly $0$ or $2$ solutions for all $a,b \in \F_{2^n}$, where $a$ is nonzero.
	\end{definition}

	There are several equivalent definitions of almost perfect nonlinear functions. We refer to \textcite{budaghyan2014} and \textcite{pott2016} for an extended overview over these functions. In this paper, we will only consider \emph{quadratic} APN functions. We define this term using the coordinate function representation of a function on $\F_2^n$.
	
	\begin{definition}
		Let $f \colon \F_2^n \to \F_2^n$, where
		\[
			f(x_1,\dots,x_n) = 
			\begin{pmatrix}
				f_1(x_1, \dots, x_n)\\\vdots\\ f_n(x_1, \dots, x_n)
			\end{pmatrix}
		\]
		for Boolean coordinate functions $f_1, \dots, f_n\colon \F_2^n \to \F_2$. The maximal degree of the coordinate functions $f_1, \dots, f_n$ is called the \emph{algebraic degree} of $f$. We call a function of algebraic degree $2$ \emph{quadratic}, and a function of algebraic degree $1$ \emph{affine}. If $f$ is affine and has no constant term, we call $f$ \emph{linear}.
	\end{definition}
	
	In polynomial mapping representation, any quadratic function $f$ on $\F_{2^n}$ can be written in the form
	\[
		f(x) = \sum_{\substack{i,j = 0 \\ i \le j}}^{n-1}\alpha_{i,j} x^{2^i+2^j} + \sum_{i = 0}^{n-1}\beta_i x^{2^i} + \gamma,
	\]
	and any affine function $f \colon \F_{2^n} \to \F_{2^n}$ can be written as
	\[
		f(x) = \sum_{i=0}^{n-1}\beta_i x^{2^i} + \gamma.
	\]
	If $f$ is affine and $\gamma = 0$, then $f$ is linear. Similar terms are used to describe polynomials over $\F_{2^n}$. Denote by $\F_{2^n}[X]$ the univariate polynomial ring over $\F_{2^n}$. A polynomial of the form
	\[
		P(X) = \sum_{i\ge 0}\alpha_i X^{2^i}
	\]
	is called a \emph{linearized polynomial}. Note that there is a one-to-one correspondence between linear functions on $\F_2^n$ and linearized polynomials in $\F_{2^n}[X] / (X^{2^n}-X)$.  In the same way as for univariate polynomials, we define a linearized polynomial in the multivariate polynomial ring~$\F_{2^n}[X_1,\dots, X_r]$ as a polynomial of the form
	\[
		P(X_1,\dots,X_r) = \sum_{j=1}^{r} \left(\sum_{i\ge 0}\alpha_{i,j} X_j^{2^i}\right).
	\]
	
	We will use such polynomials to study the equivalence of APN functions. In this paper, we are interested in \emph{inequivalent} APN functions. There are several notions of equivalence between vectorial Boolean functions that preserve the APN property. We list them in the following definition.
	
	\begin{definition}
	\label{def:equivalence}
		Two functions $f,g \colon \F_{2^n} \to \F_{2^n}$ are called 
		\begin{itemize}
			\item \emph{Carlet-Charpin-Zinoviev equivalent} (CCZ-equivalent), if there is an affine permutation $C$ on $\F_{2^n} \times \F_{2^n}$ such that
			\[
			C(G_f) = G_g,
			\]
			where $G_f = \{(x,f(x)) : x \in \F_{2^n}\}$ is the graph of $f$,
			\item \emph{extended affine equivalent} (EA-equivalent) if there exist three affine functions $A_1,A_2,A_3 \colon \F_{2^n} \to \F_{2^n}$, where $A_1$ and $A_2$ are permutations, such that
			\[
			f(A_1(x)) = A_2(g(x)) + A_3(x),
			\]
			\item \emph{affine equivalent} if they are extended affine equivalent and $A_3(x) = 0$,
			\item \emph{linearly equivalent} if they are affine equivalent and $A_1, A_2$ are linear.
		\end{itemize}
	\end{definition}

	CCZ-equivalence is the most general known notion of equivalence that preserves the APN property. Obviously, linear equivalence implies affine equivalence, and affine equivalence implies EA-equivalence. Moreover, it is well known that EA-equivalence implies CCZ-equivalence but, in general, the converse is not true. For quadratic APN~functions, however, \textcite{yoshiara2012} proved that also the converse holds.
	
	\begin{proposition}[{\textcite[Theorem~1]{yoshiara2012}}]
	\label{prop:yoshiara}
		Let $f$ and $g$ be quadratic APN functions on a finite field $\F_{2^n}$ with $n \ge 2$. Then $f$ is CCZ-equivalent to $g$ if and only if $f$ is EA-equivalent to $g$.
	\end{proposition}
	
	In this paper, \autoref{prop:yoshiara} will allow us to prove the CCZ-inequivalence of certain quadratic APN functions by showing that they are EA-inequivalent.\par 
	
	We will often consider functions on vector spaces of even dimension~$n = 2m$. Such functions can be represented in a \emph{bivariate} description as a map on $\Fpm^2 := \Fpm \times \Fpm$ with two coordinate functions. In this case, we will describe EA-equivalence as follows: Two functions $f,g \colon \Fpm^2 \to \Fpm^2$, where 
	\begin{align*}
		f(x,y) &= (f_1(x,y), f_2(x,y))& \text{and}	&&g(x,y) = (g_1(x,y), g_2(x,y))
	\end{align*}
	for coordinate functions $f_1,f_2,g_1,g_2 \colon \Fpm^2 \to \Fpm$, are EA-equivalent, if there exist affine functions $L,N,M \colon \Fpm^2 \to \Fpm^2$, where $L$ and $N$ are bijective, such that
	\[
		f(L(x,y)) = N(g(x,y)) + M(x,y).
	\]
	Write 
	\[
		L(x,y) = (L_A(x,y), L_B(x,y)) \quad \mbox{and} \quad M(x,y) = (M_A(x,y), M_B(x,y))
	\] 
	for affine functions $L_A, L_B, M_A, M_B \colon \Fpm^2 \to \Fpm$ and 
	\[
		N(x,y) = \left(N_1(x) + N_3(y),\ N_2(x) + N_4(y)\right)
	\] 
	for affine functions $N_1, \dots, N_4 \colon \Fpm \to \Fpm$. In terms of these newly defined functions, $f$ and $g$ are EA-equivalent if both
	\begin{align}
	\label{eq:LinEquiv_1}
		f_1(L_A(x,y), L_B(x,y)) &= N_1(g_1(x,y)) + N_3(g_2(x,y)) + M_A(x,y),\\
	\label{eq:LinEquiv_2}
		f_2(L_A(x,y), L_B(x,y)) &= N_2(g_1(x,y)) + N_4(g_2(x,y)) + M_B(x,y)
	\end{align}
	hold. They are affine equivalent if $M(x,y) = 0$, and they are linearly equivalent if $M(x,y) = 0$ and the functions $L$ and $N$ are linear. Studying EA-equivalence, the constants of the affine functions that determine the equivalence can be omitted as they only lead to a shift in the input and in the output. Hence, we will usually consider the functions $L_A, L_B$, $M_A, M_B$, $N_1,\dots,N_4$ as linear functions and describe them as linearized polynomials in the respective polynomial ring. Equations \cref{eq:LinEquiv_1} and \cref{eq:LinEquiv_2} will form the general framework in the proof of our main theorem.\par
	
	Not only will we solve equivalence problems in this paper, but we will also present the size of the automorphism group of several vectorial Boolean functions.
	
	\begin{definition}
	\label{def:automorphism_group}
		Let $f$ be a vectorial Boolean function on $\F_{2^n}$. We define the \emph{automorphism group of $f$ under CCZ-equivalence} as the group of affine permutations on $\F_{2^n} \times \F_{2^n}$ that preserve the graph of $f$. We denote this automorphism group by $\Aut(f)$. We analogously define the \emph{automorphism group $\Aut_{EA}(f)$ of $f$ under EA-equivalence} and the \emph{automorphism group $\Aut_L(f)$ of $f$ under linear equivalence} as the groups of the respective equivalence mappings on $\F_{2^n} \times \F_{2^n}$.
	\end{definition} 

	Note that for CCZ-, EA- and affine equivalence, the automorphism group of $f$ can be also interpreted as the automorphism group of an associated code, see e.\,g. \textcite[Section~7]{edel2009}. Most importantly, for any vectorial Boolean function $f$ on $\F_{2^n}$, the automorphism group $\Aut(f)$ is isomorphic to the automorphism group of the code
	\[
		\begin{pmatrix}1\\x\\f(x)\end{pmatrix},
	\]
	where $x \in \F_2^n$. Regarding the automorphism groups of APN functions, we need the following two lemmas. The first one follows from \citeauthor{yoshiara2012}'s~\cite{yoshiara2012} proof of \autoref{prop:yoshiara}.
	
	\begin{lemma}
	\label{lem:CCZ_EA}
		Let $f$ be a quadratic APN function on the finite field $\F_{2^n}$. Then
		\[
		\Aut(f) = \Aut_{EA}(f).
		\]
	\end{lemma}
	
	The next result follows from the definitions of the different notions of equivalence in \autoref{def:equivalence}.
	\begin{lemma}
	\label{lem:AutomorphismGroup}
		Denote by $(\F_{2^n},+)$ the additive group of the finite field $\F_{2^n}$. Let $f$ be a function on $\F_{2^n}$. Then
		\[
		\Aut_{EA}(f) = \left(\F_{2^n},+\right) \rtimes \Aut_L(f).
		\]
	\end{lemma}
	
	To conclude this section, we state some well-known results from elementary number theory that we will use regularly throughout the paper. We summarize them in the following lemma.
	\begin{lemma}
	\label{lem:num_theory}
		\begin{enumerate}[label=(\alph*)]
			\item Let $m$ be even. Then
				\[
					2^m-1 \equiv 0 \pmod 3.
				\]
			\item Let $k$ and $m$ be integers. Then
				\[
					\gcd(2^k-1, 2^m-1) = 2^{\gcd(k,m)} - 1.
				\]
			\item Let $m$ be an even integer and let $k$ be an integer coprime to $m$. Then
				\[
					\gcd(2^k+1, 2^m-1) = 3.
				\]
		\end{enumerate}
	\end{lemma}
	
\section{Known classes of APN functions}
\label{sec:known_classes}
	In this section, we give a short overview over the currently known APN functions. We will additionally motivate why we choose the Pott-Zhou APN functions out of all the known APN functions to establish a lower bound on the number of APN functions.\par
	
	In \autoref{tab:powerfunctions}, we present the known APN power functions. 
	\begin{table}[]
		\centering
		\caption{List of known APN power functions $x \mapsto x^d$.}
		\label{tab:powerfunctions}
		\begin{tabularx}{\textwidth}{XXX}
			\hline
			&Exponents $d$	&Conditions\rule[-.5em]{0em}{1.5em}\\\hline
			Gold functions		&$2^i+1$		&$\gcd(i,n) = 1,\ i \le \lfloor \frac{n}{2} \rfloor$\rule[-.5em]{0em}{1.5em}\\
			Kasami functions		&$2^{2i}-2^i+1$	&$\gcd(i,n) = 1,\ i \le \lfloor \frac{n}{2} \rfloor$\rule[-.5em]{0em}{0em}\\
			Welch function		&$2^k+3$		&$n = 2k+1$\rule[-.5em]{0em}{0em}\\
			Niho function		&$2^k+2^{\frac{k}{2}}-1$, $k$ even	& $n=2k+1$\rule[-.5em]{0em}{0em}\\
			&$2^k+2^{\frac{3k+1}{2}}-1$, $k$ odd	&	$n=2k+1$\rule[-.5em]{0em}{0em}\\
			Inverse function		&$2^{2k}-1$		&$n=2k+1$\rule[-.5em]{0em}{0em}\\
			Dobbertin function	&$2^{4k} + 2^{3k} + 2^{2k} + 2^{k}-1$	&$n=5k$\rule[-.5em]{0em}{0em}\\\hline
		\end{tabularx}
	\end{table}
	This list is conjectured to be complete. Studying the functions from this list, it is obvious that none of these classes provides plenty of inequivalent functions as there are simply not enough possible choices for the relevant parameters. Hence, APN power functions are not well suited to establish a good lower bound on the total number of inequivalent APN functions. Nevertheless, APN power functions and their equivalence relations are very well studied. It is well known that the classes in \autoref{tab:powerfunctions} are in general CCZ-inequivalent. Moreover, it is, for example, known that Gold functions are inequivalent for different values of $i$. In \autoref{sec:Gold_APN}, we will take a careful look at the equivalence relations between distinct Gold functions as they will play an important role in the proof of our main theorem. \par

	As far as non-power APN functions are concerned, the situation becomes much less clear than for power functions. Several infinite families of non-power APN functions have been found, but not much is known about their equivalence relations. This includes equivalence relations both between functions from different classes as well as between functions coming from the same class. Recently, \textcite{budaghyan2019} actually reduced the number of known classes of non-power APN functions by proving that several of them coincide. The authors present an updated list \cite[Table~3]{budaghyan2019} of known quadratic APN functions that are CCZ-inequivalent to power functions which contains nine distinct classes.\par
	
	In this paper, we focus on the family~(F10) from this list. It was introduced in 2013 by Pott and the second author~\cite{zhou2013}. In \autoref{th:ZhouPottAPN}, we restate their construction in bivariate representation, which was also used in the original paper. In the list by \textcite{budaghyan2019}, the function is given in univariate polynomial representation.
	
	\begin{theorem}[{\cite[Corollary~2]{zhou2013} and \cite[Proposition~3.5]{anbar2019}}]
	\label{th:ZhouPottAPN}
		Let $m$ be even and let $k,s$ be integers, $0 \le k,s \le m$, such that $k$ is coprime to $m$. Let $\alpha \in \Fpm^*$. The function $f_{k,s,\alpha}:\Fptwom \to \Fptwom$ defined as
		\[
			f_{k,s,\alpha}(x,y) = \left(x^{2^k+1} + \alpha y^{(2^k+1)2^s},\ xy\right)
		\]
		is APN if and only if $s$ is even and $\alpha$ is a non-cube.
	\end{theorem}
	
	Pott and the second author~\cite{zhou2013} showed that the restrictions on the parameters $s$ and $\alpha$ in \autoref{th:ZhouPottAPN}, namely on $s$ to be even and on $\alpha$ to be a non-cube, are sufficient for the function to be APN. It was recently proved by \textcite{anbar2019} that these conditions are also necessary. In \autoref{lem:TrivialEquivalences}, we will show that if $k$ and $s$ are fixed, the functions $f_{k,s,\alpha}$ are linearly equivalent for different choices of $\alpha$. Thus, we will omit the subscript $\alpha$ in the future and simply denote the Pott-Zhou APN~function by $f_{k,s}$.
	
	We chose the Pott-Zhou APN functions to study the equivalence problem for several reasons: first, in comparison to the other functions from the list by \textcite[Table~3]{budaghyan2019}, the family from \autoref{th:ZhouPottAPN} is remarkable as it depends on two parameters---recall that $\alpha$ is irrelevant. Moreover, both those parameters $k$ and $s$ are integers, and they have to meet conditions that are relatively easy to handle. Second, in the first coordinate function of the bivariate representation of~$f_{k,s}$, the Gold function~$x \mapsto x^{2^k+1}$ occurs twice, and Gold functions are well-studied. Third, Pott and the second author~\cite{zhou2013} showed that $f_{k,s}$ is a planar function on the finite field $\F_{p^{2m}}$, where $p$ is odd. Most importantly, the authors solved the equivalence problem for these planar functions and thereby gave us a starting point to solve our equivalence problem for the case $p=2$.\par
	For all those reasons, the Pott-Zhou APN family seems to be destined as a candidate to establish a lower bound on the total number of inequivalent APN functions.

\section{On the equivalence of Gold APN functions}
\label{sec:Gold_APN}
	
	Before we prove our main theorem in \autoref{sec:ZP_APN}, we state a well-known result about the equivalence of Gold APN functions in \autoref{th:GoldAPNs}. We present a new proof for this result which allows us to determine the precise shape of the equivalence mappings of Gold APN functions. We will need these equivalence mappings for the proof of \autoref{th:ZP_inequivalence1}. Note that Gold functions are quadratic, hence, by \autoref{lem:CCZ_EA} and \autoref{lem:AutomorphismGroup}, two Gold functions are CCZ-equivalent if and only if they are EA-equivalent, and their automorphism groups under CCZ- and EA-equivalence are the same.\par
	
	Our new proof shows that, for $m \ge 5$, the automorphisms of Gold APN functions are monomials. The case $m=4$ will be considered separately in \autoref{lem:GoldAPNs_m=4}. Note that for a Gold APN function $x \mapsto x^{2^k+1}$ on $\F_{2^m}$, it is easy to see that it is linearly equivalent to the function $x \mapsto x^{2^{-k}+1}$. Hence, we will only consider Gold APN functions with $k < \frac{m}{2}$.
	
	\begin{theorem}
	\label{th:GoldAPNs}
		Let $m \ge 5$, and let $k,\ell$ be integers coprime to $m$ such that $0 < k,\ell < \frac{m}{2}$. Two Gold APN functions $f,g \colon \Fpm \to \Fpm$ where
		\begin{align*}
			f(x) = x^{2^k+1} \qquad\text{and}\qquad g(x) = x^{2^\ell+1}
		\end{align*}
		are CCZ-equivalent if and only if $k = \ell$. In this case, the functions are linearly equivalent, and the equation $f(L(x)) = N(g(x))$ holds for all $x \in \Fpm$ if and only if $L(X)$ and $N(X)$ are linearized monomials of the shapes $L(X) = a_u X^{2^u}$ and $N(X) = a_u^{2^k+1} X^{2^u}$.
	\end{theorem}
	\begin{proof}
		If $k=\ell$, the functions $f$ and $g$ are clearly EA-equivalent and thereby CCZ-equivalent. We will show that, if the APN functions $f$ and $g$ are EA-equivalent, it follows that $k=\ell$. Assume that $f$ and $g$ are EA-equivalent. Then there exist three linearized polynomials~$L(X),N(X),M(X) \in \Fpm[X]$, where $N(X)$ and $L(X)$ are permutation polynomials, such that
		\begin{align}
		\label{eq:Gold}
			(L(x))^{2^k+1} = N(x^{2^\ell+1}) + M(x)
		\end{align}
		for all $x \in \Fpm$. Let $x \in \Fpm$. By writing $L(X) = \sum_{i=0}^{m-1}a_iX^{2^i}$ and $N(X) = \sum_{i=0}^{m-1} b_iX^{2^i}$ and rearranging the left side of~\cref{eq:Gold}, we obtain
		\begin{align}
		\label{eq:Gold1}
			\sum_{i=0}^{m-1} a_{i-k}^{2^k} a_i x^{2^{i+1}} + \sum_{\substack{i,j=0,\\j \ne i+k}}^{m-1} a_i^{2^k} a_j x^{2^{i+k}+2^j} = \sum_{i=0}^{m-1} b_i x^{(2^{\ell}+1)2^i} + M(x).
		\end{align}
		Since the first sum on the left-hand side of \cref{eq:Gold1} is a linearized polynomial and the second sum on the left-hand side does not include any linear parts, it follows that 
		\begin{align}
		\label{eq:Gold_M}
			M(X) = \sum_{i=0}^{m-1} a_{i-k}^{2^k} a_i X^{2^{i+1}}.
		\end{align}
		We store this information and will not consider $M(X)$ in the following steps. Rewrite the second sum on the left-hand side of \cref{eq:Gold1} as
		\[
			\sum_{0 \le i < j \le m-1} \left(a_{i-k}^{2^k} a_j + a_{j-k}^{2^k}a_i\right) x^{2^i+2^j},
		\]
		where the subscripts of $a$ are calculated modulo $m$. From \cref{eq:Gold1}, it follows that 
		\begin{align}
			\label{eq:GoldCoeff1}
			a_{i-k}^{2^k} a_{i+\ell} + a_{i+\ell-k}^{2^k} a_i	&= b_i	&&\text{for all } i, \\
			\label{eq:GoldCoeff2}
			a_{i-k}^{2^k} a_j + a_{j-k}^{2^k} a_i 		&= 0	&&\text{for } j \ne i, i \pm \ell.	
		\end{align}
		Since $L(X)$ is a permutation polynomial, assume $b_u \ne 0$ for some $u \in \{0,\dots,m-1\}$. Then, by \cref{eq:GoldCoeff1}, $a_{u-k}$ and $a_u$ cannot be zero at the same time. We will consider the two cases that, first, exactly one of $a_{u-k}$ and $a_u$ is nonzero and, second, both $a_{u-k}$ and $a_u$ are nonzero.
		
		\paragraph{Case 1.} Assume $a_{u-k} \ne 0$ and $a_u =0$. We will show that, in this case, $f$ and $g$ are equivalent if $k = \ell$ and that $L(X)$ and $N(X)$ are monomials. For the case $a_u \ne 0$ and $a_{u-k} = 0$, an analogous result can be obtained by following the same steps as in the case presented.\par
		
		If $a_{u-k} \ne 0$ and $a_u=0$, then \eqref{eq:GoldCoeff2} becomes
		\[
			a_{u-k}^{2^k}a_j + a_{j-k}^{2^k}a_u = a_{u-k}^{2^k}a_j = 0 \qquad \text{for } j \ne u, u \pm \ell.
		\]
		Consequently, $a_j = 0$ for $j \ne u, u \pm \ell$. So, only $a_{u-k}, a_{u-\ell}, a_{u+\ell}$ can be nonzero and \cref{eq:Gold1} is now
		\begin{align}
		\label{eq:GoldCase1}
			\left(a_{u-k}x^{2^{u-k}} + a_{u-\ell}x^{2^{u-\ell}} + a_{u+\ell}x^{2^{u+\ell}}\right)^{2^k+1} = \sum_{i=0}^{m-1} b_i x^{2^i(2^\ell+1)}.
		\end{align}
		The left-hand side of \cref{eq:GoldCase1} contains the summands
		\begin{align*}
			a_{u-k}^{2^k+1}x^{2^{u-k}(2^k+1)},&& a_{u-\ell}^{2^k+1}x^{2^{u-\ell}(2^k+1)}&& \text{and} && a_{u+\ell}^{2^k+1}x^{2^{u+\ell}(2^k+1)}.
		\end{align*}

		For $\ell \ne k$, it follows from the condition $0 < k, \ell < \frac{m}{2}$, that $u-k \not\equiv u \pm \ell \pmod{m}$. Hence, the sum of $a_{u-k}^{2^k+1}x^{2^{u-k}(2^k+1)}$ and one of the other two expressions cannot be $0$, which means $a_{u-k}^{2^k+1}x^{2^{u-k}(2^k+1)}$ cannot be canceled from the left-hand side of \cref{eq:GoldCase1}. However, as it cannot occur on the right-hand side, we obtain $a_{u-k} = 0$ which contradicts our assumption.\par

		For $\ell = k$, however, we obtain 
		\begin{align}
		\label{eq:GoldCase1a}
			\left(a_{u-k}x^{2^{u-k}} + a_{u+k}x^{2^{u+k}}\right)^{2^k+1} = \sum_{i=0}^{m-1} b_i x^{2^i(2^k+1)}
		\end{align}
		from \cref{eq:Gold1}, whose left-hand side can be written as
		\[
			a_{u-k}^{2^{k+1}} x^{2^{u-k}(2^k+1)} + a_{u+k}^{2^{k+1}} x^{2^{u+k}(2^k+1)} +
			a_{u-k}^{2^k} a_{u+k} x^{2^u(2^k+1)} +
			a_{u+k}^{2^k} a_{u-k}x^{2^{u-k}(2^{3k}+1)}.
		\]
		As $m \ge 5$, we have $3k \not\equiv \pm k \pmod{m}$, hence, $x^{2^{u-k}(2^{3k}+1)}$ cannot be represented in the form $x^{2^i(2^k+1)}$. Consequently, its coefficient $a_{u+k}^{2^k} a_{u-k}$ has to be $0$. Since $a_{u-k} \ne 0$, this implies that $a_{u+k} = 0$ and $L(X) = a_{u-k}X^{2^{u-k}}$ is a monomial. Thus, also $N(X)$ is a monomial. It is uniquely determined by $L(X)$ and can be written as $N(X) = a_{u-k}^{2^k+1}X^{2^{u-k}}$. Furthermore, it follows from \cref{eq:Gold_M} that $M(X)=0$.
		
		\paragraph{Case 2.} Assume both $a_u, a_{u-k} \ne 0$. First, let $k \ne \ell$. Since $0 < k,\ell < \frac{m}{2}$, it follows that $u-k \not\equiv u-\ell \pmod{m}$ and $u+\ell-k \not\equiv u \pmod{m}$. Hence, we consider \cref{eq:GoldCoeff2} for $i=u$ and $j = u-k$:
		\[
			a_{u-k}^{2^k+1} + a_{u-2k}^{2^k}a_u = 0.
		\]
		Consequently, $a_{u-2k} \ne 0$. Now, by considering \cref{eq:GoldCoeff2} for $(i,j) = (u-k, u-2k), (u-2k, u-3k) \dots, (u-(m-1)k, u)$ and recalling that $\gcd(k,m) = 1$, we obtain $a_i \ne 0$ for all $i = 0, \dots, m-1$. Moreover, it follows from this sequence of equations that the quotient
		\[
			\frac{a_{i-k}^{2^k}}{a_i} = \frac{a_{u-k}^{2^k}}{a_u} =: \Delta 
		\]
		is constant for all $i = 0, \dots, m-1$. However, consider \cref{eq:GoldCoeff1} for $i=u$:
		\begin{align}
		\label{eq:i=u}
			a_{u-k}^{2^k} a_{u+\ell} + a_{u+\ell-k}^{2^k} a_u	&= b_u.
		\end{align}
		If we divide \cref{eq:i=u} by the nonzero $a_ua_{u+\ell}$, we obtain
		\[
			\frac{a_{u-k}^{2^k}}{a_u} + \frac{a_{u+\ell-k}^{2^k}}{a_{u+\ell}} =  \frac{b_u}{a_ua_{u+\ell}}
		\]
		which is a contradiction as the left-hand side is $\Delta + \Delta = 0$ and the right-hand side is nonzero.
		
		Now, let $k=\ell$. In this case, \cref{eq:GoldCoeff1} becomes
		\begin{align}
		\label{eq:GoldCase2}
			a_{u-k}^{2^k}a_{u+k} + a_u^{2^k+1} = b_u
		\end{align}
		for $i = u$. We consider \cref{eq:GoldCoeff2} for $i = u-k$ and $j = u+k$:
		\[
			a_{u-2k}^{2^k}a_{u+k} + a_u^{2^k}a_{u-k} = 0.
		\]
		Recall that $a_{u-k},a_u \ne 0$, thus $a_{u-2k}, a_{u+k} \ne 0$. From additionally considering~\cref{eq:GoldCoeff2} for $(i,j) = (u-2k, u), (u-3k, u-k), \dots, (u, u+2k)$, it follows that $a_i \ne 0$ for all $i = 0, \dots, m-1$. Furthermore, we obtain from these equations that
		\[
			\frac{a_{u-ik}^{2^k}}{a_{u-(i+1)k}} =
			\begin{cases}
				\dfrac{a_{u}^{2^k}}{a_{u+k}} =: \Delta_1 & \text{for } i \text{ even},\rule[-1.8em]{0em}{0em}\\
				\dfrac{a_{u-k}^{2^k}}{a_{u}} =: \Delta_2 & \text{for } i \text{ odd}.
			\end{cases} 
		\]
		Note that $u-k\not\equiv u+3k \pmod{m}$ as $m \ge 5$. Hence, considering \cref{eq:GoldCoeff2} with $i = u$ and $j = u+3k$, we obtain that
		\[
			a_{u-k}^{2^k}a_{u+3k} + a_{u+2k}^{2^k}a_u = 0,
		\]
		which implies $\Delta_1 = \Delta_2 =: \Delta$. If we now divide \cref{eq:GoldCase2} by $a_u a_{u+k}$, we obtain the same kind of contradiction as for $k \ne \ell$. Hence, this second case does not provide additional solutions for $L(X),N(X)$ and $M(X)$.
	\end{proof}
	
	If $m=4$, some of the arguments used in the previous proof do not hold. In this case, there is only one Gold APN function with $k < \frac{m}{2}$, namely $f(x) = x^3$. However, unlike for $m \ge 5$, the automorphism group of $f$ does not only consist of monomials, as we show in the following lemma.
	\begin{lemma}
	\label{lem:GoldAPNs_m=4}
		The group $\Aut_L(f)$ of linear automorphisms of the unique Gold APN~function $f\colon\F_{2^4} \to \F_{2^4}$ that is defined as $f(x)=x^3$ consists of the linearized monomials from \autoref{th:GoldAPNs} together with the linearized polynomials
		\begin{align*}
			L(X) &= a_1 X^2 + a_3 X^8	&\text{and}	&&N(X) &= a_3^2 a_1 X + a_1^3 X^2 + a_1^2 a_3 X^4 + a_3^3 X^8,\\
			L(X) &= a_0 X + a_2 X^{4}	&\text{and}	&&N(X) &= a_0^3 X + a_0^2 a_2 X^2 + a_2^3 X^4 + a_2^2 a_0 X^8,
		\end{align*}
		for coefficients $a_1,\dots, a_4 \in \Fpm^*$ such that $\frac{a_1}{a_3}$ and $\frac{a_0}{a_2}$ are non-cubes. 
	\end{lemma}
	\begin{proof}
		Let $m = 4$. Using the same notation as in the proof of \autoref{th:GoldAPNs}, we consider~\cref{eq:GoldCoeff1} and \cref{eq:GoldCoeff2} for $k = \ell = 1$. We obtain the following equations of type \cref{eq:GoldCoeff1}: 
		\begin{align}
		\label{eq:Gold_m=4_bi}
			a_3^2 a_1 + a_0^3 &= b_0,&
			a_0^2 a_2 + a_1^3 &= b_1,&
			a_1^2 a_3 + a_2^3 &= b_2,&
			a_2^2 a_0 + a_3^3 &= b_3.
		\end{align}
		Note that we now only have two equations of type \cref{eq:GoldCoeff2}, namely
		\begin{align}
		\label{eq:Gold_m=4_0}
			a_1^2 a_0 + a_3^2 a_2 &= 0& \text{and}&&	a_0^2 a_3 + a_2^2 a_1 &= 0.
		\end{align}
		Assume again that $b_u \ne 0$ for some $u \in \{0,\dots, 3\}$. We will distinguish the same cases as in the proof of \autoref{th:GoldAPNs}: 
		\paragraph{Case 1.} First, assume $a_{u-1} \ne 0$ and $a_u = 0$. As before, the case $a_{u} \ne 0$ and $a_{u-1} = 0$ can be treated analogously. If $a_{u-1} \ne 0$ and $a_u = 0$, it follows from \cref{eq:Gold_m=4_0} that $a_{u+1} \ne 0$ and $a_{u-2} = 0$. However, unlike in the proof of \autoref{th:GoldAPNs}, we do not obtain a contradiction from \cref{eq:GoldCase1a} now, as $x^{2^{3k}+1} = x^9$ can be written as $x^{2^{3k}(2^k+1)} = x^{2^3\cdot 3}$. Hence, the equation
		\[
			L(x)^3 = N(x^3)
		\] 
		holds not only for the linearized monomials from \autoref{th:GoldAPNs}, but also for the linearized polynomials
		\begin{align*}
			L(X) &= a_1 X^2 + a_3 X^8	&\text{and}	&&N(X) &= a_3^2 a_1 X + a_1^3 X^2 + a_1^2 a_3 X^4 + a_3^3 X^8
		\end{align*}
		if we choose $u = 0$ or $2$ and
		\begin{align*}
			L(X) &= a_0 X + a_2 X^{4}	&\text{and}	&&N(X) &= a_0^3 X + a_0^2 a_2 X^2 + a_2^3 X^4 + a_2^2 a_0 X^8
		\end{align*}
		if we choose $u=1$ or $3$.
		
		In the final step of this first case, we need to check under which conditions $L(X)$ and $N(X)$ are permutation polynomials. Since $L(X)$ and $N(X)$ are linearized, it is sufficient to show that $L(x) = 0$ and $N(x) = 0$ if and only if $x = 0$. For $x \ne 0$, the equation $L(x) = 0$ can be rearranged to $\frac{a_1}{a_3} = x^6$ and $\frac{a_0}{a_2} = x^3$, respectively. These equations have no solution if and only if $\frac{a_1}{a_3}$ and $\frac{a_0}{a_2}$ are non-cubes. It is routine to verify that $N(X)$ also is a permutation polynomial in these cases.\par
		
		\paragraph{Case 2.} Now, let both $a_{u-1},a_u \ne 0$. In this case, it follows from \cref{eq:Gold_m=4_0} that $a_1,\dots,a_4$ are nonzero, and that $\frac{a_1}{a_3}$ and $\frac{a_2}{a_0}$ have to be cubes satisfying $(\frac{a_1}{a_3})^2 = \frac{a_2}{a_0}$. Consequently, by~\cref{eq:Gold_m=4_bi}, the coefficients $b_1,\dots,b_4$ are also nonzero which implies $a_0^3 \ne a_3^2 a_1$. Taking all these conditions into consideration, we obtain 15 choices for $a_{1}$, five choices for $a_{3}$, twelve choices for $a_{0}$, and $a_2$ is finally uniquely determined by the other coefficients. Thus, we obtain a total of $900$ possible distinct polynomials~$L(X)$. However, it can be verified that none of them is a permutation polynomial. Hence, the second case does not provide additional solutions.
	\end{proof}
	
	From \autoref{th:GoldAPNs} and \autoref{lem:GoldAPNs_m=4}, we easily obtain the automorphism groups of Gold APN functions. These results were originally given by \textcite[Proposition~5]{berger1996} in a coding theory context. We restate their result to demonstrate how it can be derived from \autoref{th:GoldAPNs}.
	
	\begin{corollary}
	\label{cor:GoldAPN_AutomorphismGroup}
		Let $f$ be a Gold APN function on $\Fpm$. If $m \ge 5$, then $\Aut_L(f)$ is isomorphic to the general semi-linear group $\GamL(1,2^m)$ of degree $1$ over $\F_{2^m}$, and consequently
		\begin{align*}
			|\Aut_L(f)| &= m(2^m-1)& \text{and} &&|\Aut(f)| &= m 2^m (2^m-1).
		\end{align*}
		If $m = 4$, then
		\begin{align*}
			|\Aut_L(f)| &= 360& \text{and} &&|\Aut(f)| &= 5760.
		\end{align*}
	\end{corollary}
	\begin{proof}
		In \autoref{th:GoldAPNs}, we have shown that equivalence mappings of Gold APN functions can be described by polynomials of the shape $L(X) = a_uX^{2^u}$ and $N(X) = a_u^{2^k+1}X^{2^u}$. We count the number of such polynomials. For all $m \ge 4$, there exist $m (2^m-1)$ distinct monomials $L(X)$, as there are $m$ distinct choices for $u$ and $2^m-1$ possibilities to choose~$a_{u}$. The monomial $N(X)$ is then uniquely determined by $L(X)$. Hence, clearly, $\Aut_L(f)$ is isomorphic to the semilinear group $\GamL(1,2^m)$ and $|\Aut_L(f)| = m (2^m-1)$.\par
		If $m=4$, in addition to the $4 \cdot (2^4-1) = 60$ monomials from \autoref{th:GoldAPNs}, we have the linearized polynomials presented in \autoref{lem:GoldAPNs_m=4}. For both possible shapes of $L(X)$ and $N(X)$ from \autoref{lem:GoldAPNs_m=4}, we have 15 possible choices for $a_{u-k}$ resulting in ten choices for $a_{u+k}$. This gives us a total number of $300 = 2 \cdot 15 \cdot 10$ distinct pairs of $L(X)$ and $N(X)$. Adding this number to the number of distinct monomials, we obtain $|\Aut_L(f)| = 60 + 300 = 360$.\par
		In both of the above cases, the automorphism group $\Aut(f)$ is obtained from $\Aut_L(F)$ by using \autoref{lem:AutomorphismGroup} in combination with \autoref{lem:CCZ_EA}.
	\end{proof}
	\textcite{berger1996} actually showed that for $m=4$, the automorphism group $\Aut_L(f)$ is isomorphic to the general semilinear group $\GamL(2,4)$. 
	
\section{On the equivalence of Pott-Zhou APN functions}
\label{sec:ZP_APN}

	In this section, we study the equivalence problem of the Pott-Zhou APN functions on~$\Fptwom$, where $m$ is even, which were introduced in \autoref{th:ZhouPottAPN}. We will answer the question for which values of the parameters $k,s,\alpha$ two Pott-Zhou APN functions $f_{k,s,\alpha}$ are CCZ-inequivalent. Our results will allow us to establish a lower bound on the total number of CCZ-inequivalent APN functions on $\Fptwom$, where $m$ is even.\par
	
	Note that, like the Gold APN functions in \autoref{sec:Gold_APN}, Pott-Zhou APN functions are quadratic. Hence, by \autoref{lem:CCZ_EA} and \autoref{lem:AutomorphismGroup}, two Pott-Zhou APN functions are CCZ-equivalent if and only if they are EA-equivalent, and their automorphism groups under CCZ- and EA-equivalence are the same. We begin by proving some trivial equivalences:
	\begin{lemma}
	\label{lem:TrivialEquivalences}
		Let $m$ be an even integer. Let $k,\ell$ be integers coprime to $m$ such that $0 < k,\ell < m$, and let $s,t$ be even integers with $0 \le s,t \le m$. Let $\alpha, \beta \in \Fpm^*$ be non-cubes. The two APN functions $f_{k,s,\alpha},f_{\ell,t,\beta} \colon \Fptwom \to \Fptwom$ from \autoref{th:ZhouPottAPN} are linearly equivalent
		\begin{enumerate}[label=(\alph*),ref=(\alph*)]
			\item\label{item:non-cubics} if $(k,s) = (\ell,t)$, no matter which non-cubes $\alpha$ and $\beta$ we choose,
			\item\label{item:k=-l,s=-t} if $k \equiv \pm \ell \pmod{m}$ and $s \equiv \pm t\pmod{m}$.
		\end{enumerate}
	\end{lemma}
	\begin{proof}
		By \cref{eq:LinEquiv_1} and \cref{eq:LinEquiv_2}, the two functions $f_{k,s,\alpha}$ and $f_{\ell,t,\beta}$ are linearly equivalent if there exist invertible mappings $L,N$ on $\F_{2^m}^2$, represented by linearized polynomials $L_A(X,Y), L_B(X,Y) \in \Fpm[X,Y]$ and $N_1(X),\dots,N_4(X) \in \Fpm[X]$, respectively, such that the two equations
		\begin{align*}
			L_A(x,y)^{2^k+1} + \alpha L_B(x,y)^{(2^k+1)2^s} &= N_1(x^{2^\ell+1}+\beta y^{(2^\ell+1)2^t}) + N_3(xy),\\
			L_A(x,y)L_B(x,y) &= N_2(x^{2^\ell+1}+\beta y^{(2^\ell+1)2^t}) + N_4(xy)
		\end{align*}
		hold for all $x,y \in \Fpm$. Note that in all the following cases, $N_2(X) = N_3(X) = 0$, hence we will not mention these polynomials in the remainder of the proof. Let $x,y \in \Fpm$.
		
		\begin{enumerate}[label=(\alph*)]
			\item Let $(k,s) = (\ell,t)$, and denote by $\gamma$ a primitive element of $\Fpm$. For the non-cubes~$\alpha, \beta \in \Fpm^*$ write $\alpha = \gamma^a$ and $\beta = \gamma^b$ for some $a,b \not\equiv 0 \pmod{3}$. We will distinguish the two cases $a \equiv b \pmod{3}$ and $a \not\equiv b \pmod{3}$. First, assume $a \equiv b \pmod{3}$. Then $f$ and $g$ are linearly equivalent as can be seen by choosing
			\begin{align*}
				L_A(X,Y) &= x,& L_B(X,Y) &= \gamma^cY,& N_1(X)&=X,&  N_4(X) &= \gamma^cX,
			\end{align*}
			where $c \in \{0, \dots, 2^m-1\}$ such that 
			\[
				(2^k+1) 2^s c \equiv b-a \pmod{2^m-1}.
			\]
			Such an integer $c$ always exists as $\gcd((2^k+1)2^s,2^m-1) = 3$ and $b-a \equiv 0 \pmod{3}$. If $a \not\equiv b \pmod{3}$, then $f$ and $g$ are linearly equivalent by
			\begin{align*}
				L_A(X,Y) &= X^2,& L_B(X,Y)&=\gamma^cY^2,& N_1(X) &= X^2,&  N_4(X) &= \gamma^cX^2,
			\end{align*}
			where $c$ satisfies
			\[
				(2^k+1) 2^s c \equiv 2b-a \pmod{2^m-1}.
			\]
			By the same reasoning as before and considering that $2b-a \equiv 0 \pmod{3}$, such an integer $c$ always exists.
			
			\item By \ref{item:non-cubics}, we can assume $\alpha = \beta$. We write $f_{k,s}$ instead of $f_{k,s,\alpha}$. If $k \equiv -\ell \pmod{m}$ and $s=t$, then $f_{k,s}$ and $f_{-k,s}$ are linearly equivalent by
			\begin{align*}
				L_A(X,Y) &= X^{2^{-k}},& L_B(X,Y)&=Y^{2^{-k}},& N_1(X) &= X,&  N_4(X) &= X^{2^{-k}}.
			\end{align*}
			
			Now, let $k=\ell$ and $s \equiv -t \pmod{m}$. Define a function $f'\colon \Fpm \to \Fpm$ as
			\[
				f'(x,y) = \left(y^{2^k+1} + \alpha^{2^{-s}}x^{(2^k+1)2^s},\ xy\right).
			\]
			We show that both $f_{k,s}$ and $f_{k,-s}$ are linearly equivalent to $f'$. For $f_{k,s}$, we choose 
			\begin{align*}
				L_A(X,Y) &= Y,& L_B(X,Y)&=X,& N_1(X)&=X,& N_4(X) &= X
			\end{align*} 
			and use the result from \ref{item:non-cubics}. The function $f_{k,-s}$ is linearly equivalent to $f'$ by
			\begin{align*}
				L_A(X,Y) &= X,& L_B(X,Y) &= Y,& N_1(X) &= \alpha^{2^{-s}}X^{2^s},& N_4(X) &= X.
			\end{align*}
			
			As $f_{k,s}$ is linearly equivalent to both $f_{-k,s}$ and $f_{k,-s}$, it follows from combining these results that $f_{k,s}$ is linearly equivalent to $f_{-k,-s}$.
		\end{enumerate}
		\vspace{-2em}
	\end{proof}
	
	Thanks to \autoref{lem:TrivialEquivalences}, we will, from now on, fix the non-cube $\alpha$ and restrict the parameters $k,s$ to $0 < k < \frac{m}{2}$ and $0 \le s \le \frac{m}{2}$. We will moreover omit the subscript $\alpha$ and simply write $f_{k,s}$ instead of $f_{k,s,\alpha}$.\par
	
	We are now able to prove \autoref{th:ZP_inequivalence1}. In its proof, we only consider Pott-Zhou APN functions on the finite field $\Fptwom$ with $m \ge 6$. Note that, for $m=2$, all the Pott-Zhou APN functions are CCZ-equivalent according to \autoref{lem:TrivialEquivalences}. Hence, up to equivalence, the function $f_{1,0}$ is the unique Pott-Zhou APN function on $\F_{2^2}$. It is actually equivalent to the unique Gold APN function $x \mapsto x^3$. The result for the case $m=4$ was already given by Pott and the second author~\cite{zhou2013}, we restate it at the beginning of our proof. Let us now recall \autoref{th:ZP_inequivalence1}.
	
	\begin{thm:ZP}
		Let $m\ \ge 4$ be an even integer. Let $k,\ell$ be integers coprime to $m$ such that $0 < k,\ell < \frac{m}{2}$, let $s,t$ be even integers with $0 \le s,t \le \frac{m}{2}$, and let $\alpha,\beta \in \Fpm^*$ be non-cubes. Two Pott-Zhou APN~functions $f_{k,s,\alpha}, f_{\ell,t,\beta} \colon \Fptwom \to \Fptwom$ from \autoref{th:ZhouPottAPN}, where
		\begin{align*}
			f_{k,s,\alpha}(x,y) = \left(x^{2^k+1} + \alpha y^{(2^k+1)2^s},\ xy\right)	&&\text{and}	&& f_{\ell,t,\beta}(x,y) = \left(x^{2^\ell+1} + \beta y^{(2^\ell+1)2^t},\ xy\right),
		\end{align*}
		are CCZ-equivalent if and only if $k = \ell$ and $s = t$.
	\end{thm:ZP}
	\begin{proof}
		As shown in \autoref{lem:TrivialEquivalences}, the choice of the non-cubes $\alpha$ and $\beta$ does not matter. Hence, set $\alpha = \beta$ and denote the two Pott-Zhou APN functions by $f_{k,s}$ and $f_{\ell,t}$.\par
		For $m=4$, there exist two equivalence classes of Pott-Zhou APN functions, namely $f_{1,0}$ and $f_{1,2}$. Their CCZ-inequivalence has been shown by Pott and the second author~\cite{zhou2013} who computed their $\Gamma$-ranks as $13200$ and $13642$, respectively.\par
		For the remainder of this proof, let $m \ge 6$. As in the proof of \autoref{lem:TrivialEquivalences}, the functions $f_{k,s}$ and $f_{\ell,t}$ are EA-equivalent if there exist linearized polynomials $L_A (X,Y)$, $L_B(X,Y)$, $M_A(X,Y)$, $M_B(X,Y) \in \Fpm[X,Y]$ and $N_1(X), \dots, N_4(X) \in \F_{2^m}[X]$, where 
		\[
			L(X,Y) = (L_A(X,Y), L_B(X,Y))
		\]
		and
		\[
			N(X,Y) = (N_1(X) + N_3(Y),\ N_2(X) + N_4(Y))
		\]
		are invertible, such that the equations
		\begin{align}
		\label{eq:ZhouPottEquiv1}
			L_A(x,y)^{2^k+1} + \alpha L_B(x,y)^{(2^k+1)2^s} &= N_1(x^{2^\ell+1}+\alpha y^{(2^\ell+1)2^t}) + N_3(xy) + M_A(x,y),\\
		\label{eq:ZhouPottEquiv2}
			L_A(x,y)L_B(x,y) &= N_2(x^{2^\ell+1}+\alpha y^{(2^\ell+1)2^t}) + N_4(xy) + M_B(x,y)
		\end{align}
		hold for all $x,y \in \Fpm$. We write $L_A(X,Y) = L_1(X) + L_3(Y)$ and $L_B(X,Y) = L_2(X) + L_4(Y)$ for linearized polynomials $L_1(X), \dots, L_4(X) \in \F_{2^m}[X]$. Hence,
		\[
			L(X,Y) = \left(L_1(X)+L_3(Y),\ L_2(X) + L_4(Y)\right).
		\]
		Write 
		\begin{align*}
			L_1(X) = \sum_{i=0}^{m-1} a_i X^{2^i},&& L_2(X) = \sum_{i=0}^{m-1} b_i X^{2^i},&& L_3(Y) = \sum_{i=0}^{m-1} \overline{a}_i Y^{2^i},&& L_4(Y) = \sum_{i=0}^{m-1} \overline{b}_i Y^{2^i}.
		\end{align*}
		Moreover, define linearized polynomials $M_1(X), \dots, M_4(X) \in \F_{2^m}[X]$ in the same way as $L_1(X), \dots, L_4(X)$ were defined. For the remainder of the proof, let $x,y \in \Fpm$. We first prove the following claim.
		
		\paragraph{Claim.} If $f_{k,s}$ and $f_{\ell,t}$ are EA-equivalent, then $k=\ell$ and each of the linearized polynomials $L_1(X), \dots, L_4(X)$ is either a binomial, a monomial or zero.\\
		
		\noindent We will prove the result for $y=0$, hence we only consider $L_1(X)$ and $L_2(X)$. By proceeding analogously, it can be shown that the statement also holds for $x=0$ and the polynomials $L_3(Y)$ and $L_4(Y)$. Let $y=0$. Then \cref{eq:ZhouPottEquiv1} and \cref{eq:ZhouPottEquiv2} can be reduced to
		\begin{align}
		\label{eq:y=0_1}
			L_1(x)^{2^k+1} + \alpha L_2(x)^{(2^k+1)2^s} &= N_1(x^{2^\ell+1}) + M_1(x),\\
		\label{eq:y=0_2}
			L_1(x) L_2(x) &= N_2(x^{2^\ell+1}) + M_2(x)
		\end{align}
		for all $x \in \Fpm$. Write 
		\begin{align*}
			N_1(X) = \sum_{i=0}^{m-1} c_i X^{2^i}&& \text{and} &&N_2(X) = \sum_{i=0}^{m-1} d_i X^{2^i}.
		\end{align*}
		
		We first consider the case, that one of $L_1(X)$ or $L_2(X)$ is zero. Assume $L_1(X) \ne 0$ and $L_2(X)=0$. If $L_1(X)=0$ and $L_2(X) \ne 0$, the same result can be obtained by symmetry. In our case, \cref{eq:y=0_1} becomes
		\[
			L_1(x)^{2^{k+1}} = N_1(x^{2^\ell+1}) + M_1(x),
		\]
		which is equivalent to the statement that the Gold APN functions $x \mapsto x^{2^k+1}$ and $x \mapsto x^{2^\ell+1}$ are EA-equivalent. According to \autoref{th:GoldAPNs}, this holds if and only if $k=\ell$. We showed that, in this case, $L_1(X)$ is a linearized monomial.\par
		
		Now, let both $L_1(X), L_2(X) \ne 0$. Then \cref{eq:y=0_2} becomes
		\begin{align}
		\label{eq:y=0_Case2_1}
			\sum_{i=0}^{m-1}a_ib_i x^{2^{i+1}} + \sum_{\substack{i,j=0,\\j \ne i}}^{m-1}a_ib_j x^{2^i+2^j} = \sum_{i=0}^{m-1}d_ix^{(2^\ell+1)2^i} + M_2(x).
		\end{align}
		The first sum on the left-hand side of \cref{eq:y=0_Case2_1} is a linearized polynomial and the second sum does not contain linear parts. Hence, $M_2(x) = \sum_{i=0}^{m-1}a_ib_i x^{2^{i+1}}$. Omitting the linear parts, we rewrite \cref{eq:y=0_Case2_1} as
		\[
			\sum_{0 \le i < j \le m-1} (a_ib_j + a_jb_i)x^{2^i+2^j} = \sum_{i=0}^{m-1}d_ix^{(2^\ell+1)2^i}
		\]
		and obtain
		\begin{align}
		\label{eq:y=0_Case2_LinEq1}
			a_i b_{i+\ell} + a_{i+\ell} b_i 	&= d_i 	&&\text{for all } i,\\
		\label{eq:y=0_Case2_LinEq2}
			a_i b_j + a_j b_i 					&=0		&&\text{for } j \ne i, i\pm \ell,
		\end{align}
		where the subscripts are calculated modulo $m$. We separate the proof into two cases: First, the case that $d_i = 0$ for all $i = 0, \dots, m-1$ and, second, the case that $d_u \ne 0$ for some $u \in \{0,\dots,m-1\}$. 
		
		\subparagraph{Case 1.} In this case, we show that if $d_i=0$ for all $i = 0, \dots, m-1$, the problem can be reduced to the Gold APN Case from \autoref{th:GoldAPNs}, and hence, $k = \ell$ and $L_1(X)$ and $L_2(X)$ are monomials of the same degree. Assume $d_i=0$ for all $i = 0, \dots, m-1$, which means $N_2(X) = 0$. In this case, \cref{eq:y=0_Case2_LinEq1} and \cref{eq:y=0_Case2_LinEq2} combine to
		\begin{align}
		\label{eq:y=0_Case2_LinEq3}
			a_i b_j + a_j b_i 					&=0		&&\text{for } j \ne i.
		\end{align}
		Recall, that $L_1(X), L_2(X) \ne 0$. Consequently, there are at least two nonzero coefficients~$a_u, b_{u'}$. If $u = u'$, then the corresponding term $a_u b_u X^{2^{u+1}}$ is linearized. Hence, it is a part of~$M_2(X)$, not of $N_2(X)$. If $u \ne u'$, then, by \cref{eq:y=0_Case2_LinEq3}, 
		\[
			a_u b_{u'} + a_{u'} b_u = 0.
		\]
		Hence, $a_{u'},b_{u} \ne 0$ and $\frac{a_u}{b_u} = \frac{a_u'}{b_u'}$. Moreover, it follows that all pairs $(a_j, b_j)$ satisfy either
		\begin{align}
			\label{eq:y=0_TypeI_TypeII}
			a_j = b_j &= 0	&&\text{or}	&& \frac{a_j}{b_j} = \Delta,
		\end{align}
		where $\Delta := \frac{a_u}{b_u}$ is a nonzero constant. Consequently, $b_j = \delta a_j$, where $\delta = \Delta^{-1}$, for all $j=0, \dots, m-1$, and $L_2(X)$ is a multiple of $L_1(X)$, namely 
		\[
			L_2(X) = \delta L_1(X).
		\]
		Written in this way, it is obvious that $L_1(X)L_2(X) = \delta (L_1(X))^2$ is a linearized polynomial, hence $N_2(X) = 0$ and $M_2(X) = \delta (L_1(X))^2$. Next, we plug $L_1(X)$ and $L_2(X)$ into \cref{eq:y=0_1} and obtain
		\begin{align}
		\label{eq:y=0_N_2(x)=0}
			L_1(x)^{2^k+1} + \alpha\delta^{(2^k+1)2^s} L_1(x)^{(2^k+1)2^s} = N_1(x^{2^\ell + 1}) + M_1(x).
		\end{align}
		
		If $s = 0$, then \cref{eq:y=0_N_2(x)=0} becomes
		\[
			\left(1+\alpha\delta^{(2^k+1)2^s}\right) L_1(x)^{2^k+1} = N_1(x^{2^\ell + 1}) + M_1(x),
		\]
		which implies that the Gold APN functions $x \mapsto x^{2^k+1}$ and $x \mapsto x^{2^\ell +1}$ are EA-equivalent. According to \autoref{th:GoldAPNs}, it follows that $k = \ell$ and that $L_1(X)$ is a monomial. Consequently, $L_2(X) = \delta L_1(X)$ is also a monomial, it has the same degree as $L_1(X)$.\par
		
		If $s \ne 0$, we define a mapping $P\colon \F_{2^m} \to \F_{2^m}$ by
		\[
			P(x) = x + \alpha \delta^{(2^k+1)2^s}x^{2^s}
		\]
		and rewrite the left hand side of \cref{eq:y=0_N_2(x)=0} as
		\[
			P(L_1(x)^{2^k+1}).
		\]
		
		We show that $P$ is bijective. Since $P$ is linear, it is sufficient to show that it has no nonzero roots. If $P$ had a nonzero root, it would solve the equation
		\[
			\alpha^{-1} = \delta^{(2^k+1)2^s}x^{2^s-1}.
		\]
		However, this equation can never be true: its left-hand side is obviously a non-cube. Since $\gcd(2^k+1, 2^m-1) = 3$, the first factor on the right-hand side, $\delta^{(2^k+1)2^s}$, is a cube. As $\gcd(2^s-1,2^m-1) = 2^{\gcd(s,m)}-1 = 2^{2\gcd(\frac{s}{2},\frac{m}{2})}-1$ is divisible by $3$, the second factor, $x^{2^s-1}$, is also a cube. Hence, we have a cube on the right-hand side and a non-cube on the left-hand side, which is a contradiction.\par
		
		Denote by $P^{-1}$ the inverse of $P$ and rewrite \cref{eq:y=0_N_2(x)=0} as
		\begin{align}
		\label{eq:y=0_N_2(x)=0_1}
			L_1(x)^{2^k+1} = P^{-1}(x) \circ N_1(x^{2^\ell+1}) + P^{-1}(x) \circ M_1(x).
		\end{align}
		Note that $P^{-1}$ is also linear. Hence, \cref{eq:y=0_N_2(x)=0_1} leads us to the Gold APN function case again, and it follows that $k = \ell$ and that $L_1(X)$ and $L_2(X)$ are monomials of the same degree.
				
		\subparagraph{Case 2.} Consider \cref{eq:y=0_Case2_LinEq1} and \cref{eq:y=0_Case2_LinEq2} again. In this case, we show that if $d_u \ne 0$ for some~$u \in \{0, \dots, m-1\}$, we obtain that $k=\ell$ and that $L_1(X)$ and $L_2(X)$ have one of the following shapes: either
		\begin{align*}
			L_1(X) = a_u X^u	&&\text{and} &&L_2(X) = b_u X^u
		\end{align*}
		or
		\begin{align*}
			L_1(X) = a_u X^u	&&\text{and} &&L_2(X) = b_{u+k} X^{u+k}
		\end{align*}
		or
		\begin{align*}
			L_1(X) = a_u X^u + a_{u+k} X^{u+k}	&&\text{and} &&L_2(X) = b_u X^u + b_{u+k} X^{u+k}.
		\end{align*}
		
		Assume $d_u \ne 0$ for some $u \in \{0,\dots, m-1\}$ which means $N_2(X) \ne 0$. Then, by \cref{eq:y=0_Case2_LinEq1}, $a_u$ and $b_u$ cannot be zero at the same time. We will separate the proof of Case 2 into two subcases: first, Case 2.1, where both $a_u$ and $b_u$ are nonzero, and second, Case 2.2, where exactly one of $a_u$ and $b_u$ is nonzero. Both these cases will be separated into several subcases again.\par
		
		\subparagraph{Case 2.1.} Assume $a_u \ne 0$ and $b_u \ne 0$. Then, from \cref{eq:y=0_Case2_LinEq2}, it follows that all pairs~$(a_j,b_j)$, where $j \ne u, u \pm \ell$, satisfy \cref{eq:y=0_TypeI_TypeII}. We will separate the proof of this case into three subcases:
		
		\subparagraph{Case 2.1.1.} Assume there exists $\ell' \ne 0, \pm\ell, \pm 2\ell$ such that $a_{u+\ell'}, b_{u+\ell'} \ne 0$. By \cref{eq:y=0_TypeI_TypeII}, this implies $\frac{a_{u+\ell'}}{b_{u+\ell'}} = \Delta$. Since $u+\ell' \pm \ell \ne u \pm \ell$, it follows from \cref{eq:y=0_Case2_LinEq2} with $i =u + \ell'$ that both $(a_{u+\ell},b_{u+\ell})$ and $(a_{u-\ell},b_{u-\ell})$ also have to satisfy one of the equations in \cref{eq:y=0_TypeI_TypeII}. Hence, \cref{eq:y=0_TypeI_TypeII} holds for all pairs $(a_j,b_j)$, and we know from the calculations below \cref{eq:y=0_TypeI_TypeII} that this implies $N_2(X) = 0$. This is a contradiction.
		
		\subparagraph{Case 2.1.2.} Now, assume $a_j=b_j=0$ for $j \ne u,u\pm \ell$.
		In this case, we obtain only one equation from \cref{eq:y=0_Case2_LinEq2}, namely 
		\[
			a_{u-\ell}b_{u+\ell} + a_{u+\ell}b_{u-\ell} = 0.
		\]
		Hence, either
		\begin{enumerate}[label=(\roman*), ref=(\roman*)]
			\item\label{item:Case2Ab_1} $a_{u-\ell} = a_{u+\ell} = 0$ or $b_{u-\ell} = b_{u+\ell} = 0$, meaning that one of $L_1(X)$ and $L_2(X)$ is a monomial and the other one a trinomial, or
			\item\label{item:Case2Ab_2} $a_{u-\ell} = b_{u-\ell} = 0$ or $a_{u+\ell} = b_{u+\ell} = 0$, meaning that both $L_1(X)$ and $L_2(X)$ are binomials consisting of terms of the same degree, or
			\item\label{item:Case2Ab_3} $a_{u \pm \ell},b_{u \pm \ell} \ne 0$ and $\frac{a_{u-\ell}}{b_{u-\ell}} = \frac{a_{u+\ell}}{b_{u+\ell}}$, meaning that both $L_1(X)$ and $L_2(X)$ are trinomials.
		\end{enumerate}
	  	We will consider each of these three subcases.
	  	
  		\subparagraph{Subcase~\ref{item:Case2Ab_1}.} Assume $b_{u-\ell} = b_{u+\ell} = 0$. The case $a_{u-\ell} = a_{u+\ell} = 0$ follows by symmetry. We consider polynomials
  		\begin{align*}
	  		L_1(X) &= a_{u-\ell}X^{2^{u-\ell}} + a_u X^{2^u} + a_{u+\ell}X^{2^{u+\ell}} & \text{and}&& L_2(X) &= b_u X^{2^u}
  		\end{align*}
  		which we plug into the left-hand side of \cref{eq:y=0_1}. Hence,
  		\begin{align}
  		\nonumber
	  		L_1(x)^{2^k+1} &=
	  		a_{u-\ell}^{2^k+1} x^{2^{u-\ell}(2^k+1)} +
	  		a_{u}^{2^k+1} x^{2^{u}(2^k+1)} +
	  		a_{u+\ell}^{2^k+1} x^{2^{u+\ell}(2^k+1)} \\&\quad\label{eq:y=0_TypeI_TypeIIb_L1}+
	  		a_{u-\ell}^{2^k}a_u x^{2^{u}(2^{k-\ell}+1)} +
	  		a_{u}^{2^k}a_{u+\ell} x^{2^{u+\ell}(2^{k-\ell}+1)} +
	  		a_{u+\ell}^{2^k}a_{u-\ell} x^{2^{u-\ell}(2^{k+2\ell}+1)} \\&\quad\nonumber+
	  		a_{u-\ell}^{2^k}a_{u+\ell} x^{2^{u+\ell}(2^{k-2\ell}+1)} +
	  		a_{u}^{2^k}a_{u-\ell} x^{2^{u-\ell}(2^{k+\ell}+1)} +
	  		a_{u+\ell}^{2^k}a_{u} x^{2^{u}(2^{k+\ell}+1)}.
  		\end{align}
  		and
  		\begin{align}
  		\label{eq:y=0_TypeI_TypeIIb_L2}
  			\alpha L_2(x)^{2^s(2^k+1)} = \alpha b_u^{2^s(2^k+1)}x^{2^{s+u}(2^k+1)}.
  		\end{align}
		Recall that the right-hand side of \cref{eq:y=0_1} is 
		\[
			\sum_{i=0}^{m-1}c_i x^{2^i(2^\ell+1)} + M_1(x).
		\]
		We show, that not all of the first three summands of \cref{eq:y=0_TypeI_TypeIIb_L1}, that all contain the factor~$(2^k+1)$ in their exponents, can be canceled. As $0 < \ell < \frac{m}{2}$, they cannot cancel each other. If $\ell = \frac{m}{2}-k$, the exponent of the sixth term contains the factor $(2^k+1)$, it can be written as $2^{u-\frac{m}{2}}(2^k+1)$. However, by the same reasoning as above, it cannot cancel any of the first three terms. The only case where one summand could be canceled is the following: if $\ell=k$, the seventh and the second term can be summarized and could potentially cancel each other. In total, for arbitrary $k$ and $\ell$, at least the first and the third summand of~\cref{eq:y=0_TypeI_TypeIIb_L1} contain $(2^k+1)$ in their exponents. Note that none of them can be canceled by~\cref{eq:y=0_TypeI_TypeIIb_L2}: as $m$ and $s$ are even and $\gcd(\ell,m) =1$, it follows that $s \not\equiv \pm \ell \pmod{m}$.\par
		
		We now compare the left-hand side and the right-hand side of \cref{eq:y=0_1}. Since the left-hand side contains terms with $x^{2^i(2^k+1)}$, it follows that $k=\ell$. Note that in this case, the fourth and fifth summand of \cref{eq:y=0_TypeI_TypeIIb_L1} become linearized, hence
		\[
			M_1(X) = a_{u-k}^{2^k}a_u X^{2^{u+1}} +
			a_{u}^{2^k}a_{u+k} X^{2^{u+k+1}}.
		\]
		Now, consider the sixth, eighth and ninth summand of \cref{eq:y=0_TypeI_TypeIIb_L1}:
		\begin{align*}
			a_{u+k}^{2^k}a_{u-k} x^{2^{u-k}(2^{3k}+1)},&&
			a_{u}^{2^k}a_{u-k} x^{2^{u-k}(2^{2k}+1)},&&
			a_{u+k}^{2^k}a_{u} x^{2^{u}(2^{2k}+1)}.
		\end{align*}
		As $m \ge 6$ and $\gcd(k,m)=1$, we have $2k \not\equiv \pm k \pmod{m}$ and  $3k \ne \pm k \pmod{m}$. Hence, these terms cannot be represented in the form $c_ix^{2^i(2^k+1)}$ which means that their coefficients have to be zero. As $a_u \ne 0$, it follows that $a_{u-k} = a_{u+k} = 0$. Hence, $L_1(X)$ and $L_2(X)$ are monomials of the same degree,
		\begin{align}
		\label{eq:y=0_L1L2_samedegree}
			L_1(X) &= a_u X^{2^u}& \text{and} && L_2(X) = b_u X^{2^u},
		\end{align}
		and $M_1(X) = 0$.
		
		\subparagraph{Subcase~\ref{item:Case2Ab_2}.} Assume $a_{u-\ell} = b_{u-\ell} = 0$. The case $a_{u+\ell} = b_{u+\ell} = 0$ follows by symmetry. In our case,
 		\begin{align*}
			L_1(X) &= a_u X^{2^u} + a_{u+\ell}X^{2^{u+\ell}} & \text{and}&& L_2(X) &= b_u X^{2^u} + b_{u+\ell}X^{2^{u+\ell}}.
		\end{align*}
		For the left-hand side of \cref{eq:y=0_1}, we obtain
  		\begin{align}
		\nonumber
		L_1(x)^{2^k+1} &=
		a_{u}^{2^k+1} x^{2^{u}(2^k+1)} +
		a_{u+\ell}^{2^k+1} x^{2^{u+\ell}(2^k+1)} \\&\quad\label{eq:y=0_TypeI_TypeIIb_L1_2}+
		a_{u}^{2^k}a_{u+\ell} x^{2^{u+\ell}(2^{k-\ell}+1)} +
		a_{u+\ell}^{2^k}a_{u} x^{2^{u}(2^{k+\ell}+1)}.
		\end{align}
		and
		\begin{align}
		\nonumber
			\alpha L_2(x)^{2^s(2^k+1)} &=
			\alpha b_u^{2^s(2^k+1)}x^{2^{s+u}(2^k+1)} +
			\alpha b_{u+\ell}^{2^s(2^k+1)}x^{2^{s+u+\ell}(2^k+1)} \\&\quad\label{eq:y=0_TypeI_TypeIIb_L2_2}+
			\alpha b_u^{2^{s+k}}b_{u+\ell}^{2^s}x^{2^{s+u+\ell}(2^{k-\ell}+1)} +
			\alpha b_{u+\ell}^{2^{s+k}}b_{u}^{2^s}x^{2^{s+u}(2^{k+\ell}+1)}.
		\end{align}
		As in Subcase \ref{item:Case2Ab_1}, the first two terms of \cref{eq:y=0_TypeI_TypeIIb_L1_2} and \cref{eq:y=0_TypeI_TypeIIb_L2_2}, respectively, cannot cancel each other. We will consider the cases $s \ne 0$ and $s = 0$.\par
		
		First, assume $s \ne 0$. As $s \not\equiv \pm \ell \pmod{m}$, the terms in \cref{eq:y=0_TypeI_TypeIIb_L1_2} and in \cref{eq:y=0_TypeI_TypeIIb_L2_2} cannot cancel each other if we add both expressions. Consequently, from comparing the left-hand side of \cref{eq:y=0_1} with its right-hand side, it follows that $k = \ell$. Using the same argument as in Subcase \ref{item:Case2Ab_1}, we obtain $a_{u+\ell} = b_{u+\ell} = 0$, and $L_1(X)$ and $L_2(X)$ are monomials of the same degree as in \cref{eq:y=0_L1L2_samedegree}. Moreover, $M_1(X) = 0$.\par
		
		Next, assume $s = 0$. Now, the corresponding terms in \cref{eq:y=0_TypeI_TypeIIb_L1_2} and \cref{eq:y=0_TypeI_TypeIIb_L2_2} can be summarized. Consider the first summand
		\begin{align}
		\label{eq:y=0_subcase3}
			\left(a_{u}^{2^k+1} + \alpha b_{u}^{2^k+1}\right) x^{2^{u}(2^k+1)}.
		\end{align}
		As $a_{u}, b_{u} \ne 0$, its coefficient is zero, if and only if
		\[
			\alpha = \left(\frac{a_{u}}{b_{u}}\right)^{2^k+1}.
		\]
		However, as $\gcd(2^k+1, 2^m-1) = 3$, this implies that $\alpha$ is a cube which is a contradiction. Hence, this term occurs with a nonzero coefficient on the left-hand side of \cref{eq:y=0_1}, and we need $k = \ell$ to represent it as $c_i x^{2^i(2^\ell + 1)}$ on the right-hand side of \cref{eq:y=0_1}. If $k=\ell$, the second term in the sum of \cref{eq:y=0_TypeI_TypeIIb_L1_2} and \cref{eq:y=0_TypeI_TypeIIb_L2_2} can also be represented in this way, and the third term is linearized which means
		\[
			M_1(X) = \left(a_u^{2^k}a_{u+k} + \alpha b_u^{2^k}b_{u+k}\right) X^{2^{u+k+1}}.
		\] 
		Hence, we consider the fourth summand:
		\[
			\left( a_{u+k}^{2^k} a_u + \alpha b_{u+k}^{2^k} b_u \right) x^{2^u(2^{2k}+1)}.
		\]
		As $2k \not\equiv \pm k \pmod{m}$, it cannot be represented as $c_i x^{2^i(2^k + 1)}$. Hence, its coefficient has to be zero. This is the case if $a_{u+k} = b_{u+k} = 0$ or if 
		\begin{align}
		\label{eq:binomial_coefficients}
			\left(\frac{a_u}{b_u}\right)\left(\frac{a_{u+k}}{b_{u+k}}\right)^{2^k} = \alpha.
		\end{align}
		Consequently, either $L_1(X)$ and $L_2(X)$ are monomials of the same degree, as in \cref{eq:y=0_L1L2_samedegree}, and $M_1(X) = 0$, or $L_1(X)$ and $L_2(X)$ are binomials of the form
 		\begin{align}
 		\label{eq:y=0_L1L2_binomials}
			L_1(X) &= a_u X^{2^u} + a_{u+k}X^{2^{u+k}} & \text{and}&& L_2(X) &= b_u X^{2^u} + b_{u+k}X^{2^{u+k}},
		\end{align}
		where the coefficients satisfy $\left(\frac{a_u}{b_u}\right)\left(\frac{a_{u+k}}{b_{u+k}}\right)^{2^k} = \alpha$. This implies $\frac{a_u}{b_u} \ne \frac{a_{u+k}}{b_{u+k}}$ since otherwise, $\alpha$ would be a cube. In the binomial case, $M_1(X) = \left(a_u^{2^k}a_{u+k} + \alpha b_u^{2^k}b_{u+k}\right) X^{2^{u+k+1}}$.

		\subparagraph{Subcase~\ref{item:Case2Ab_3}.} Now,
  		\begin{align*}
			L_1(X) &= a_{u-\ell}X^{2^{u-\ell}} + a_u X^{2^u} + a_{u+\ell}X^{2^{u+\ell}}\\ \text{and } L_2(X) &= b_{u-\ell}X^{2^{u-\ell}} + b_u X^{2^u} + b_{u+\ell}X^{2^{u+\ell}},
		\end{align*}
		where all coefficients are nonzero and $\frac{a_{u-\ell}}{b_{u-\ell}} = \frac{a_{u+\ell}}{b_{u+\ell}}$. We plug these polynomials into~\cref{eq:y=0_1}. The expression $L_1(x)^{2^k+1}$ is as in \cref{eq:y=0_TypeI_TypeIIb_L1}, and $\alpha L_2(x)^{(2^k+1)2^s}$ looks basically the same: just replace $a$ by $b$, multiply every coefficient by $\alpha$ and apply the automorphism $x \mapsto x^{2^s}$ on every summand. Furthermore, what we mentioned below \cref{eq:y=0_TypeI_TypeIIb_L2} for the coefficients of~$L_1(x)^{2^k+1}$ still holds, now for the coefficients of both $L_1(x)^{2^k+1}$ and $\alpha L_2(x)^{(2^k+1)2^s}$. As in Subcase~\ref{item:Case2Ab_2}, we separate the cases $s \ne 0$ and $s = 0$.\par 
		
		Assume $s \ne 0$. Like before, terms from $L_1(x)^{2^k+1}$ and from $\alpha L_2(x)^{(2^k+1)2^s}$ cannot cancel each other, and it follows that $k = \ell$. We obtain
		\[
			M_1(X) = 
			a_{u-k}^{2^k}a_u X^{2^{u+1}} +
			a_{u}^{2^k}a_{u+k} X^{2^{u+k+1}} +
			\alpha b_{u-k}^{2^{s+k}}b_u^{2^s} X^{2^{s+u+1}} +
			\alpha b_{u}^{2^{s+k}}a_{u+k}^{2^s} X^{2^{s+u+k+1}}.
		\]
		 By the same argument as in Subcase \ref{item:Case2Ab_1}, the sixth, eighth and ninth term of \cref{eq:y=0_TypeI_TypeIIb_L1}, that now contain $x^{2^{3k}+1}$ and $x^{2^{2k}+1}$, cannot be represented as $x^{2^i(2^k+1)}$. The same holds for the corresponding terms in $\alpha L_2(x)^{2^s(2^k+1)}$. As a consequence, the coefficients of these terms, that are
		\begin{align*}
			a_{u+k}^{2^k}a_{u-k},\ a_{u}^{2^k}a_{u-k},\ a_{u+k}^{2^k}a_{u},&& \text{and}&& \alpha b_{u+k}^{2^{k+s}}b_{u-k}^{2^s},\ \alpha b_{u}^{2^{k+s}}b_{u-k}^{2^s},\ \alpha b_{u+k}^{2^{k+s}}b_{u}^{2^s},
		\end{align*}
		have to be zero. As $a_u, b_u \ne 0$, it follows that $a_{u\pm k} = b_{u \pm k} = 0$ which contradicts our assumption.\par 
		Now, assume $s=0$. In this case, we can summarize the corresponding terms of~$L_1(x)^{2^k+1}$ and $\alpha L_2(x)^{2^s(2^k+1)}$ and obtain the same term as in \cref{eq:y=0_subcase3}. By the same argument as in Subcase~\ref{item:Case2Ab_2} for $s=0$, it follows that $k=\ell$. Now, consider the term 
		\begin{align*}
			\left(a_{u+k}^{2^k}a_{u-k} + \alpha b_{u+k}^{2^{k+s}}b_{u-k}^{2^s}\right) x^{2^{u-k}(2^{3k}+1)},
		\end{align*}
		which, as $3k \not\equiv \pm k \pmod{m}$, cannot be represented as $c_i x^{2^i(2^k+1)}$. Hence, its coefficient has to be zero. As $a_{u\pm k}$ and $b_{u\pm k}$ are nonzero, this is only the case if
		\[
			\left(\frac{a_{u-k}}{b_{u-k}}\right)\left(\frac{a_{u+k}}{b_{u+k}}\right)^{2^k} = \alpha.
		\]
		However, as $\frac{a_{u-k}}{b_{u-k}} = \frac{a_{u+k}}{b_{u+k}}$ and $\gcd(2^k+1,2^m-1) = 3$, this contradicts the condition that $\alpha$ is a non-cube. In summary, we cannot obtain possible polynomials $L_1(X)$ and $L_2(X)$ from Subcase~\ref{item:Case2Ab_3}. 
	
		\subparagraph{Case 2.1.3.} Now, assume $a_j=b_j=0$ for $j \ne u, u\pm \ell, u \pm 2\ell$. Recall that all pairs~$(a_j, b_j)$ where $j \ne u, u\pm \ell$ have to satisfy \cref{eq:y=0_TypeI_TypeII}. If $a_{u\pm2\ell} = b_{u\pm2\ell} = 0$, we are in Case~2.1.2. Hence, assume that $a_{u+2\ell}$ and $b_{u+2\ell}$ are nonzero. One can obtain an almost identical result by symmetry when assuming that $a_{u-2\ell}$ and $b_{u-2\ell}$ are nonzero.\par
		
		If $a_{u+2\ell}, b_{u+2\ell} \ne 0$, then, by \cref{eq:y=0_TypeI_TypeII}, $\frac{a_{u+2\ell}}{b_{u+2\ell}} = \Delta$. It follows from \cref{eq:y=0_Case2_LinEq2} that also $(a_{u-2\ell},b_{u-2\ell})$ and $(a_{u-\ell}, b_{u-\ell})$ have to satisfy \cref{eq:y=0_TypeI_TypeII}. However, \cref{eq:y=0_Case2_LinEq2} does not provide any restriction on the value of $(a_{u+\ell}, b_{u+\ell})$. If $(a_{u+\ell}, b_{u+\ell})$ satisfies \cref{eq:y=0_TypeI_TypeII}, then all $(a_j,b_j)$ do and we are in Case 2.1.1. If $(a_{u+\ell}, b_{u+\ell})$ does not satisfy \cref{eq:y=0_TypeI_TypeII}, then it follows from \cref{eq:y=0_Case2_LinEq2} that $a_j=b_j=0$ for $j =u-\ell, u-2\ell$. Hence, 
		\begin{align*}
			L_1(X) &= a_u X^{2^u} + a_{u+\ell}X^{2^{u+\ell}} + a_{u+2\ell}X^{2^{u+2\ell}}\\
			\text{and } L_2(X) &= b_u X^{2^u} + b_{u+\ell}X^{2^{u+\ell}} + b_{u+2\ell}X^{2^{u+2\ell}}.
		\end{align*}
		As $\frac{a_u}{b_u} = \frac{a_{u+2\ell}}{b_{u+2\ell}}$, this case is similar to Case 2.1.2, Subcase~\ref{item:Case2Ab_3}, when we shift all coefficients by $\ell$ with the only difference that now, one of the middle coefficients $a_{u+\ell}, b_{u+\ell}$ can be zero. However, the arguments used in the previous case still hold. Consequently, we do not obtain possible polynomials $L_1(X)$ and $L_2(X)$ from Case 2.1.3.
		
		\subparagraph{Case 2.2.} Assume exactly one of $a_u$ and $b_u$ is nonzero. We show the case $a_u \ne 0$ and $b_u = 0$. The case $a_u = 0$ and $b_u \ne 0$ can be proved analogously. So, assume $a_u \ne 0$ and $b_u = 0$. From \cref{eq:y=0_Case2_LinEq1}, we obtain the equation
		\[
			a_u b_{u+\ell} + a_{u+\ell} b_u = a_u b_{u+\ell} = d_u.
		\]
		As $d_u \ne 0$, it follows that $b_{u+\ell} \ne 0$. From \cref{eq:y=0_Case2_LinEq2}, we obtain
		\begin{align*}
			a_u b_j + a_j b_u = a_u b_j = 0
		\end{align*}
		for $j \ne u, u\pm \ell$. Consequently, $b_j = 0$ for $j \ne u \pm \ell$. Now, it follows from \cref{eq:y=0_Case2_LinEq2} that
		\begin{align*}
			a_{u+\ell} b_j + a_j b_{u+\ell} = a_j b_{u+\ell} = 0
		\end{align*}
		for $j \ne u-\ell,u, u + \ell, u + 2\ell$. Consequently, $a_j = 0$ for $j \ne u-\ell,u, u + \ell, u + 2\ell$.	We will separate the proof of Case~2.2 into two subcases: in Case~2.2.1, we consider $b_{u-\ell} \ne 0$ and in Case~2.2.2, we consider $b_{u-\ell} = 0$.
		
		\paragraph{Case 2.2.1.} Assume $b_{u-\ell} \ne 0$. From \cref{eq:y=0_Case2_LinEq2}, we obtain
		\[
			a_{u-\ell} b_{u+2\ell} + a_{u+2\ell} b_{u-\ell} = a_{u+2\ell} b_{u-\ell} = 0
		\]
		which implies $a_{u+2\ell} = 0$. Moreover, we obtain
		\[
			a_{u-\ell}b_{u+\ell} + a_{u+\ell} b_{u-\ell} = 0
		\]
		which, recalling that $b_{u+\ell}$ is nonzero, implies either $a_{u-\ell} = a_{u+\ell} = 0$ or $a_{u-\ell}, a_{u+\ell} \ne 0$ and $\frac{a_{u-\ell}}{b_{u-\ell}} = \frac{a_{u+\ell}}{b_{u+\ell}}$. We separate these two subcases:
	
		\subparagraph{Subcase (i).} Assume $a_{u-\ell} = a_{u+\ell} = 0$. Then
		\begin{align*}
			L_1(X) &= a_u X^{2^u}	&\text{and}	&&L_2(X) &= b_{u-\ell}X^{2^{u-\ell}} + b_{u+\ell}X^{2^{u+\ell}}.
		\end{align*}
		We plug $L_1(X)$ and $L_2(X)$ into \cref{eq:y=0_1} and obtain on the left-hand side
		\begin{align}
		\label{eq:y=0_Case2Ba_L1}
			L_1(x)^{2^k+1} &= a_u^{2^k+1} x^{2^u(2^k+1)}
		\end{align}
		and
		\begin{align}
		\nonumber
			\alpha L_2(x)^{(2^k+1)2^s} &=
			\alpha b_{u-\ell}^{2^{s+k+1}} x^{2^{s+u-\ell}(2^k+1)} +
			\alpha b_{u+\ell}^{2^{s+k+1}} x^{2^{s+u+\ell}(2^k+1)} +\\ 
		\label{eq:y=0_Case2Ba_L2}
			&\quad \alpha b_{u-\ell}^{2^{s+k}}b_{u+\ell}^{2^s} x^{2^{s+u+\ell}(2^{k-2\ell}+1)} +
			\alpha b_{u+\ell}^{2^{s+k}}b_{u-\ell}^{2^s} x^{2^{s+u-\ell}(2^{k+2\ell}+1)}.
		\end{align}
		Recall that the right-hand side of \cref{eq:y=0_1} is 
		\[
			\sum_{i=0}^{m-1} c_i x^{2^i(2^k+1)} + M_1(x).
		\]
		Since $s \not \equiv \pm \ell \pmod{m}$, the terms containing $x^{2^k+1}$ cannot be canceled with each other. Hence, they can only be represented as $c_i x^{2^i(2^k+1)}$ if $k = \ell$. In this case, however, the last term of $\cref{eq:y=0_Case2Ba_L2}$ contains $x^{2^{3k}+1}$ which cannot be represented in the form $c_i x^{2^i(2^k+1)}$ because $m \ge 6$ and, hence, $3k \not\equiv \pm k \pmod{m}$. Consequently, the corresponding coefficient has to be zero which implies that $b_{u-\ell} = 0$ or $b_{u+\ell} = 0$. This is a contradiction.
		
		\subparagraph{Subcase (ii).} Assume $a_{u-\ell}, a_{u+\ell} \ne 0$ and $\frac{a_{u-\ell}}{b_{u-\ell}} = \frac{a_{u+\ell}}{b_{u+\ell}}$. Then
		\begin{align*}
			L_1(X) &= a_{u-\ell} X^{2^{u-\ell}} + a_u X^{2^u} +a_{u+\ell} X^{2^{u+\ell}}	&\text{and}	&&L_2(X) &= b_{u-\ell}X^{2^{u-\ell}} + b_{u+\ell}X^{2^{u+\ell}}.
		\end{align*}
		We plug these into \cref{eq:y=0_1}. Then $L_1(x)^{2^k+1}$ is as in \cref{eq:y=0_TypeI_TypeIIb_L1} and $\alpha L_2(x)^{(2^k+1)2^s}$ is as in \cref{eq:y=0_Case2Ba_L2}. Since $a_{u}^{2^k+1} x^{2^{u}(2^k+1)}$ can never be canceled by any of the terms in \cref{eq:y=0_Case2Ba_L2}, it follows that $k = \ell$. However, now the expressions $a_{u}^{2^k}a_{u-k} x^{2^{u-k}(2^{2k}+1)}$ and 
		$a_{u+k}^{2^k}a_{u} x^{2^{u}(2^{2k}+1)}$ occur on the left-hand side of \cref{eq:y=0_1}, and they cannot be represented in the form $c_ix^{2^i(2^k+1)}$ on its right-hand side. As the corresponding coefficients are nonzero, this is a contradiction.

		\subparagraph{Case 2.2.2.} Assume $b_{u-\ell} = 0$. By \cref{eq:y=0_Case2_LinEq2},
		\[
			a_{u+\ell} b_{u-\ell} + a_{u-\ell}b_{u+\ell} = a_{u-\ell}b_{u+\ell} = 0
		\]
		which implies $a_{u-\ell} = 0$. Then
		\begin{align*}
			L_1(X) &= a_u X^{2^u} + a_{u+\ell} X^{2^{u+\ell}} + a_{u+2\ell} X^{2^{u+2\ell}}	&\text{and}	&&L_2(X) &= b_{u+\ell}X^{2^{u+\ell}}.
		\end{align*}
		We plug these into \cref{eq:y=0_1}. The expression $L_1(x)^{2^k+1}$ is now similar to \cref{eq:y=0_TypeI_TypeIIb_L1}, we only need to replace $u$ by $u+\ell$. Moreover, 
		\[
			\alpha L_2(x)^{(2^k+1)2^s} = b_{u+\ell}^{2^s(2^k+1)} x^{2^{s+u+l}(2^k+1)}.
		\]
		Since $x^{2^u(2^k+1)}$ and $x^{2^{u+2\ell}(2^k+1)}$ cannot be canceled on the left-hand side of \cref{eq:y=0_1} and have to be represented on its right-hand side, it follows that $k=\ell$. However, if $k=\ell$, the summands
		\begin{align*}
			a_{u+2k}^{2^k}a_{u} x^{2^{u}(2^{3k}+1)},&&
			a_{u+k}^{2^k}a_{u} x^{2^{u}(2^{2k}+1)},&&
			a_{u+2k}^{2^k}a_{u+k} x^{2^{u+k}(2^{2k}+1)}
		\end{align*}
		on the left-hand side cannot be represented as $c_i x^{2^i(2^k+1)}$ on the right-hand side. As~$a_u \ne 0$, it follows that $a_{u+k} = a_{u+2k} = 0$. Consequently, $L_1(X)$ and $L_2(X)$ are monomials of the form
		\begin{align}
		\label{eq:y=0_L1L2_differentdegree}
			L_1(X) &= a_u X^{2^u} &\text{and}	&& L_2(X) = b_{u+k} X^{2^{u+k}},
		\end{align}
		and $M_1(X) = 0$. Note that if we consider Case 2.2 with $a_u = 0$ and $b_u \ne 0$, we obtain
		\begin{align}
		\label{eq:y=0_L1L2_differentdegree_2}
			L_1(X) &= a_{u+k} X^{2^{u+k}} &\text{and}	&& L_2(X) = b_{u} X^{2^{u}}
		\end{align}
		and $M_1(X)=0$ from Case 2.2.2. This concludes the \textbf{proof of our Claim}.\\
		
		 \noindent We summarize the results we have obtained so far. If the APN functions $f_{k,s}$ and $f_{\ell, t}$ are EA-equivalent, then $k = \ell$, and $L_1(X)$ and $L_2(X)$ are of the form
		\begin{align}
			L_1(X) &= a_u X^{2^u} + a_{u+k}X^{2^{u+k}}	&\text{and}	&& L_2(X) &= b_u X^{2^u} + b_{u+k}X^{2^{u+k}}
		\end{align}
		for some $u \in \{0,\dots,m-1\}$. If $L_1(X)$ is a binomial, then, by \cref{eq:y=0_L1L2_binomials}, $L_2(X)$ is as well. Moreover, this case is only possible if $s=0$ and the coefficients of $L_1(X)$ and $L_2(X)$ satisfy \cref{eq:binomial_coefficients}. If $L_1(X)$ is a monomial, then, by \cref{eq:y=0_L1L2_samedegree}, \cref{eq:y=0_L1L2_differentdegree}, \cref{eq:y=0_L1L2_differentdegree_2} and \autoref{th:GoldAPNs}, $L_2(X)$ is a monomial or zero. If $L_1(X) = 0$, then, by \autoref{th:GoldAPNs}, $L_2(X)$ is a monomial.\par
		
		Vice versa, the same statements hold for $L_3(Y)$ and $L_4(Y)$, where 
		\begin{align}
			L_3(Y) &= \overline{a}_w Y^{2^w} + \overline{a}_{w+k}Y^{2^{w+k}}	&\text{and}	&& L_4(Y) &= \overline{b}_w Y^{2^w} + \overline{b}_{w+k}Y^{2^{w+k}}
		\end{align}
		for some $w \in \{0,\dots,m-1\}$.\par
		
		It remains to show that EA-equivalence of $f_{k,s}$ and $f_{k,t}$ implies $s = t$. Combining the results on $L_1(X), \dots, L_4(X)$ that we mentioned above, we need $L_A(X,Y)$ and $L_B(X,Y)$ to be of one of the following forms:
		
		\begin{enumerate}[label=\textbf{(\alph*)}, ref=\textbf{(\alph*)}]
			\item \label{item:a} $L_A(X,Y) = a_u X^{2^u} + a_{u+k}X^{2^{u+k}} + \overline{a}_w Y^{2^w} + \overline{a}_{w+k}Y^{2^{w+k}}$\\
			and $L_B(X,Y) = b_u X^{2^u} + b_{u+k}X^{2^{u+k}} + \overline{b}_w Y^{2^w} + \overline{b}_{w+k}Y^{2^{w+k}}$,
			
			\item \label{item:b} $L_A(X,Y) = a_u X^{2^u} + a_{u+k}X^{2^{u+k}} + \overline{a}_w Y^{2^w}$\\
			and $L_B(X,Y) = b_u X^{2^u} + b_{u+k}X^{2^{u+k}} + \overline{b}_w Y^{2^w}$,
			
			\item \label{item:c} $L_A(X,Y) = a_u X^{2^u} + a_{u+k}X^{2^{u+k}} + \overline{a}_w Y^{2^w}$\\
			and $L_B(X,Y) = b_u X^{2^u} + b_{u+k}X^{2^{u+k}} + \overline{b}_{w+k}Y^{2^{w+k}}$,
			
			\item \label{item:d} $L_A(X,Y) = a_u X^{2^u} + a_{u+k}X^{2^{u+k}} + \overline{a}_{w+k} Y^{2^{w+k}}$\\
			and $L_B(X,Y) = b_u X^{2^u} + b_{u+k}X^{2^{u+k}} + \overline{b}_{w}Y^{2^{w}}$,
			
			\item \label{item:e} $L_A(X,Y) = a_u X^{2^u} + \overline{a}_w Y^{2^w} + \overline{a}_{w+k}Y^{2^{w+k}}$\\
			and $L_B(X,Y) = b_u X^{2^u} + \overline{b}_w Y^{2^w} + \overline{b}_{w+k}Y^{2^{w+k}}$,
			
			\item \label{item:f} $L_A(X,Y) = a_u X^{2^u} + \overline{a}_w Y^{2^w} + \overline{a}_{w+k}Y^{2^{w+k}}$\\
			and $L_B(X,Y) = b_{u+k} X^{2^{u+k}} + \overline{b}_w Y^{2^w} + \overline{b}_{w+k}Y^{2^{w+k}}$,
			
			\item \label{item:g} $L_A(X,Y) = a_{u+k} X^{2^{u+k}} + \overline{a}_w Y^{2^w} + \overline{a}_{w+k}Y^{2^{w+k}}$\\
			and $L_B(X,Y) = b_{u} X^{2^{u}} + \overline{b}_w Y^{2^w} + \overline{b}_{w+k}Y^{2^{w+k}}$,
			
			\item \label{item:h} $L_A(X,Y) = a_u X^{2^u} + \overline{a}_w Y^{2^w}$ and $L_B(X,Y) = b_u X^{2^u} + \overline{b}_w Y^{2^w}$,
			
			\item \label{item:i} $L_A(X,Y) = a_u X^{2^u} + \overline{a}_w Y^{2^w}$ and $L_B(X,Y) = b_u X^{2^u} + \overline{b}_{w+k} Y^{2^{w+k}}$,
			
			\item \label{item:j} $L_A(X,Y) = a_u X^{2^u} + \overline{a}_{w+k} Y^{2^{w+k}}$ and $L_B(X,Y) = b_u X^{2^u} + \overline{b}_w Y^{2^w}$,
			
			\item \label{item:k} $L_A(X,Y) = a_u X^{2^u} + \overline{a}_w Y^{2^w}$ and $L_B(X,Y) = b_{u+k} X^{2^{u+k}} + \overline{b}_w Y^{2^w}$,
			
			\item \label{item:l} $L_A(X,Y) = a_{u+k} X^{2^{u+k}} + \overline{a}_w Y^{2^w}$ and $L_B(X,Y) = b_u X^{2^u} + \overline{b}_w Y^{2^w}$,
			
			\item \label{item:m} $L_A(X,Y) = a_u X^{2^u} + \overline{a}_w Y^{2^w}$ and $L_B(X,Y) = b_{u+k} X^{2^{u+k}} + \overline{b}_{w+k} Y^{2^{w+k}}$,
			
			\item \label{item:n} $L_A(X,Y) = a_u X^{2^u} + \overline{a}_{w+k} Y^{2^{w+k}}$ and $L_B(X,Y) = b_{u+k} X^{2^{u+k}} + \overline{b}_w Y^{2^w}$,
			
			\item \label{item:o} $L_A(X,Y) = a_{u+k} X^{2^{u+k}} + \overline{a}_{w} Y^{2^{w}}$ and $L_B(X,Y) = b_{u} X^{2^{u}} + \overline{b}_{w+k} Y^{2^{w+k}}$,
			
			\item \label{item:p} $L_A(X,Y) = a_{u+k} X^{2^{u+k}} + \overline{a}_{w+k} Y^{2^{w+k}}$ and $L_B(X,Y) = b_{u} X^{2^{u}} + \overline{b}_w Y^{2^w}$.
		\end{enumerate}
	
		We will show that all these cases either lead to a contradiction or to the conclusion that $L_A(X,Y)$ and $L_B(X,Y)$ need to be monomials of the same degree. Considering that EA-equivalence of $f_{k,s}$ and $f_{\ell,t}$ implies $k=\ell$, we rewrite \cref{eq:ZhouPottEquiv1} and \cref{eq:ZhouPottEquiv2} as
		\begin{align}
		\label{eq:Part2_Equiv1}
			L_A(x,y)^{2^k+1} + \alpha L_B(x,y)^{(2^k+1)2^s} &= N_1(x^{2^k+1}+\alpha y^{(2^k+1)2^t}) + N_3(xy) + M_A(x,y),\\
		\label{eq:Part2_Equiv2}
			L_A(x,y)L_B(x,y) &= N_2(x^{2^k+1}+\alpha y^{(2^k+1)2^t}) + N_4(xy) + M_B(x,y).
		\end{align}
		We will plug all the possible combinations \ref{item:a}--\ref{item:p} into these equations. Note that in cases \ref{item:a}--\ref{item:g}, $L_1(X)$ and $L_2(X)$ or $L_3(Y)$ and $L_4(Y)$ are binomials. Hence, these cases imply $s=0$, and the coefficients of the binomials have to satisfy \cref{eq:binomial_coefficients}. We moreover point out that on the right-hand side of \cref{eq:Part2_Equiv2}, the term $x^{2^i}y^{2^j}$ cannot occur if $i \not\equiv j \pmod {m}$.\par
		
		We first assume, that all the coefficients of $L_A(X,Y)$ and $L_B(X,Y)$ are nonzero. Then the cases \ref{item:c}, \ref{item:d}, \ref{item:f}, \ref{item:g}, \ref{item:i}--\ref{item:m} and \ref{item:p} lead to contradictions. We show how to obtain this contradiction for~\ref{item:i}, the reasoning for the other cases is analogous. If $L_A(X,Y)$ and $L_B(X,Y)$ are as in \ref{item:i}, the left-hand side of \cref{eq:Part2_Equiv2} contains the summands
		\begin{align*}
			a_u \overline{b}_{w+k} x^{2^u} y^{2^{w+k}} &&\text{and}&& b_{u} \overline{a}_w x^{2^u} y^{2^w}.
		\end{align*}
		 However, no matter how we choose $u$ and $w$, we can never represent $x^{2^u}y^{2^{w+k}}$ and $x^{2^u}y^{2^w}$ simultaneously in the form $x^{2^i}y^{2^i}$ on the right-hand side of \cref{eq:Part2_Equiv2}. Hence, this is a contradiction.\par
		 
		 For the remaining cases \ref{item:a}, \ref{item:b}, \ref{item:e}, \ref{item:h}, \ref{item:n}, and \ref{item:o}, however, \cref{eq:Part2_Equiv2} does not lead to a contradiction. Hence, we need to take a closer look at these.\par
		 
		 We start with \ref{item:n}, the same argumentation will also hold for \ref{item:o}: If we plug $L_A(X,Y)$ and $L_B(X,Y)$ of \ref{item:n} into \cref{eq:Part2_Equiv2}, the left-hand side is
		\begin{align*}
			a_u b_{u+k} x^{2^u(2^k+1)} + \overline{a}_w \overline{b}_{w+k} y^{2^w(2^k+1)} + a_u \overline{b}_{w} x^{2^u} y^{2^{w}} + b_{u+k} \overline{a}_{w+k} x^{2^{u+k}} y^{2^{w+k}}.
		\end{align*}
		While the first two summands can be represented as $N_2(x^{2^k+1} + \alpha y^{(2^k+1)2^t})$, we need $u=w$ for the remaining two summands to be representable by $N_4(xy)$. So, from now on, assume $u=w$. Next, we plug $L_A(X,Y)$ and $L_B(X,Y)$ into \cref{eq:Part2_Equiv1}. The left-hand side is
		\begin{align*}
			&a_u^{2^k+1} x^{2^u(2^k+1)} + \overline{a}_{u+k}^{2^k+1} y^{2^{u+k}(2^k+1)} + a_u^{2^k} \overline{a}_{u+k} x^{2^{u+k}} y^{2^{u+k}} + a_u \overline{a}_{u+k}^{2^k} x^{2^u} y^{2^{u+2k}}\\
			&\quad + \alpha b_{u+k}^{2^s(2^k+1)} x^{2^{s+u+k}(2^k+1)} + \alpha \overline{b}_u^{2^s(2^k+1)} y^{2^{s+u}(2^k+1)}\\
			&\qquad + \alpha b_{u+k}^{2^{s+k}} \overline{b}_u^{2^s} x^{2^{s+u+2k}} y^{2^{s+u}} + \alpha b_{u+k}^{2^s} \overline{b}_u^{2^{s+k}} x^{2^{s+u+k}} y^{2^{s+u+k}}.
		\end{align*}
		As neither $x^{2^u} y^{2^{u+2k}}$ nor $x^{2^{s+u+2k}} y^{2^{s+u}}$ can be represented on the right-hand side of~\cref{eq:Part2_Equiv1}, the corresponding coefficients need to be zero. Consequently, one of $a_u$ and $\overline{a}_{u+k}$ and one of $b_{u+k}$ and $\overline{b}_{u}$ have to be zero which means that $L_A(X,Y)$ and $L_B(X,Y)$ are monomials. Considering \cref{eq:Part2_Equiv2} under the assumption that $L_A(X,Y)$ and $L_B(X,Y)$ are monomials, it becomes clear that both polynomials need to be of the same degree. By symmetry, in case \ref{item:o}, $L_A(X,Y)$ and $L_B(X,Y)$ are monomials of the same degree as well.\par
		
		Next, we study \ref{item:h}: For this case, we obtain
		\begin{align*}
			a_u b_u x^{2^{u+1}} + \overline{a}_w \overline{b}_w y^{2^{w+1}} + \left(a_u \overline{b}_w + b_u \overline{a}_w\right) x^{2^u}y^{2^w}
		\end{align*}
		on the left-hand side of \cref{eq:Part2_Equiv2}. We consider two cases.\par
		
		\textbf{Case 1.} First, assume $\frac{a_u}{b_u} = \frac{\overline{a}_w}{\overline{b}_w}$. Then \cref{eq:Part2_Equiv2} does not provide any information as the left-hand side is a linearized polynomial. We plug $L_A(X,Y)$ and $L_B(X,Y)$ into \cref{eq:Part2_Equiv1}. Then the left-hand side of \cref{eq:Part2_Equiv1} contains the four summands
		\begin{align}
		\nonumber
			a_u^{2^k} \overline{a}_w x^{2^{u+k}} y^{2^w},&& \alpha b_u^{2^{k+s}} \overline{b}_w^{2^s} x^{2^{s+u+k}} y^{2^{s+w}}\\
		\label{eq:(h)_0}
			\text{and}\quad a_u \overline{a}_w^{2^k} x^{2^u} y^{2^{w+k}}, && \alpha b_u^{2^{s}} \overline{b}_w^{2^{s+k}} x^{2^{s+u}} y^{2^{s+w+k}}.
		\end{align}
		If $s \ne 0$, these terms cannot be represented on the right-hand side of \cref{eq:Part2_Equiv1}. Hence, the corresponding coefficients need to be zero which implies that $L_A(X,Y)$ and $L_B(X,Y)$ are monomials of the same degree. If $s=0$, we can summarize the terms of \cref{eq:(h)_0} to
		\begin{align}
		\label{eq:(h)_0.5}
			\left( a_u^{2^k} \overline{a}_w + \alpha b_u^{2^{k}} \overline{b}_w \right) x^{2^{u+k}} y^{2^w}&&\text{and}&& \left(a_u \overline{a}_w^{2^k} + \alpha b_u \overline{b}_w^{2^{k}}\right) x^{2^u} y^{2^{w+k}}.
		\end{align}
		The coefficients of these terms are zero if
		\begin{align}
		\label{eq:(h)_1}
			\alpha = \frac{a_u^{2^k} \overline{a}_w}{b_u^{2^k} \overline{b}_w} &&\text{and}&& \alpha = \frac{a_u \overline{a}_w^{2^k}}{b_u \overline{b}_w^{2^k}}
		\end{align}
		hold. As $\frac{a_u}{b_u} = \frac{\overline{a}_w}{\overline{b}_w}$, both equations are identical and we obtain 
		\[
			\alpha = \left(\frac{a_u}{b_u}\right)^{2^k+1}.
		\]
		However, since $\gcd(2^k+1,2^m-1) = 3$, this means that $\alpha$ is a cube. This is a contradiction. Hence, $L_A(X,Y)$ and $L_B(X,Y)$ need to be monomials of the same degree.
		
		\textbf{Case 2.} Now, assume $\frac{a_u}{b_u} \ne \frac{\overline{a}_w}{\overline{b}_w}$. Then $a_u \overline{b}_w + b_u \overline{a}_w \ne 0$, and we need $u=w$ to represent $x^{2^u}y^{2^w}$ on the right-hand side of \cref{eq:Part2_Equiv2}. Assuming $u=w$, we plug $L_A(X,Y)$ and $L_B(X,Y)$ into \cref{eq:Part2_Equiv1}. Then its left-hand side contains the summands from \cref{eq:(h)_0}, where $u = w$. As before, if $s \ne 0$, these terms cannot be represented on the right-hand side of~\cref{eq:Part2_Equiv1}. Hence, the corresponding coefficients need to be zero which implies that $L_A(X,Y)$ and $L_B(X,Y)$ are monomials of the same degree. If $s = 0$, we can summarize the terms in the same way as in \cref{eq:(h)_0.5}, where $u=w$. Their coefficients are zero if \cref{eq:(h)_1} with $u=w$ holds. This is only the case if
		\begin{align}
		\label{eq:(h)_2}
			\left( \frac{a_u \overline{b}_u}{b_u \overline{a}_u} \right)^{2^k-1} = 1.
		\end{align}
		Since $k$ and $m$ are coprime, we obtain $\gcd(2^k-1,2^m-1) = 2^{\gcd(k,m)}-1 = 1$. Consequently, \cref{eq:(h)_2} implies $\frac{a_u}{b_u} = \frac{\overline{a}_u}{\overline{b}_u}$. As $u=w$, this contradicts our assumption $\frac{a_u}{b_u} \ne \frac{\overline{a}_w}{\overline{b}_w}$.\par
		
		We next consider \ref{item:b}. Recall that, in this case, $s=0$ and, by the arguments below~\cref{eq:y=0_L1L2_binomials}, $\frac{a_u}{b_u} \ne \frac{a_{u+k}}{b_{u+k}}$. We can assume that all the coefficients of $L_A(X,Y)$ and $L_B(X,Y)$ are nonzero since otherwise, we end up in one of the cases \ref{item:h}--\ref{item:o}. If we plug $L_A(X,Y)$ and $L_B(X,Y)$ of \ref{item:b} into \cref{eq:Part2_Equiv2}, the left-hand side contains the summands
		\begin{align}
		\label{eq:(b)_1}
			\left(a_u \overline{b}_w + b_u \overline{a}_w\right) x^{2^u} y^{2^w} &&\text{and} && \left(a_{u+k} \overline{b}_w + b_{u+k} \overline{a}_w\right) x^{2^{u+k}} y^{2^w}.
		\end{align}
		It is not possible that both coefficients in \cref{eq:(b)_1} are zero at the same time: this would imply $\frac{a_u}{b_u} = \frac{a_{u+k}}{b_{u+k}}$ which is a contradiction. Since we cannot represent both summands of \cref{eq:(b)_1} simultaneously as $c_i x^{2^i} y^{2^i}$, it follows that one of the coefficients has to be zero. Assume the second one is zero. The reasoning for the case that the first one is zero can be done analogously. The second coefficient is zero if $\frac{\overline{a}_w}{\overline{b}_w} = \frac{a_{u+k}}{b_{u+k}}$. The remaining first term of~\cref{eq:(b)_1} can be only represented on the right-hand side of \cref{eq:Part2_Equiv2} if $u = w$. So, assume $u = w$ and consider \cref{eq:Part2_Equiv1}. On the left-hand side of \cref{eq:Part2_Equiv1}, we obtain the three summands
		\begin{align*}
			\left(a_u^{2^k} \overline{a}_u + \alpha b_u^{2^k} \overline{b}_u\right)x^{2^{u+k}} y^{2^u},\
			\left(a_u \overline{a}_u^{2^k} + \alpha b_u \overline{b}_u^{2^k}\right)x^{2^{u}} y^{2^{u+k}},\
			\left(a_{u+k}^{2^k} \overline{a}_u + \alpha b_{u+k}^{2^k} \overline{b}_u\right)x^{2^{u+2k}} y^{2^{u}}
		\end{align*}
		that cannot be represented on the corresponding right-hand side. Consequently, their coefficients have to be zero. The coefficient of the third term is zero if and only if 
		\[
			\alpha = \frac{a_{u+k}^{2^k} \overline{a}_u}{b_{u+k}^{2^k}\overline{b}_u}.
		\]
		However, as $\frac{\overline{a}_u}{\overline{b}_u} = \frac{a_{u+k}}{b_{u+k}}$, this implies that $\alpha$ is a cube which is a contradiction. Hence, case \ref{item:b} does not lead to additional solutions for $L_A(X,Y)$ and $L_B(X,Y)$. By the same reasoning, this also holds for case \ref{item:e}.\par
		Eventually, consider case \ref{item:a}. As in case \ref{item:b}, we can assume that all the coefficients are nonzero. If we plug $L_A(X,Y)$ and $L_B(X,Y)$ into \cref{eq:Part2_Equiv2}, the four terms
		\begin{align}
			\nonumber &\left(a_u \overline{b}_w + b_u \overline{a}_w\right) x^{2^u} y^{2^w},\
			\left(a_{u} \overline{b}_{w+k} + b_u \overline{a}_{w+k}\right) x^{2^u} y^{2^{w+k}},\\
			\label{eq:(a)_1}
			&\quad \left(a_{u+k} \overline{b}_w + b_{u+k} \overline{a}_w\right) x^{2^{u+k}} y^{2^w},\
			\left(a_{u+k} \overline{b}_{w+k} + b_{u+k} \overline{a}_{w+k}\right) x^{2^{u+k}} y^{2^{w+k}}
		\end{align}
		occur on the left-hand side of this equation. From the arguments below \cref{eq:y=0_L1L2_binomials}, we know that $\frac{a_u}{b_u} \ne \frac{a_{u+k}}{b_{u+k}}$ and  $\frac{\overline{a}_w}{\overline{b}_w} \ne \frac{\overline{a}_{w+k}}{\overline{b}_{w+k}}$. Hence, not all coefficients can be zero. In fact, only one coefficient out of each of the following pairs of coefficients in \cref{eq:(a)_1} can be zero: first and second, third and fourth, first and third, second and fourth. Since $x^{2^u} y^{2^{w+k}}$ and $x^{2^{u+k}}y^{2^w}$ cannot be represented simultaneously as $c_i x^{2^i}y^{2^i}$ on the right-hand side of \cref{eq:Part2_Equiv2}, the case that both the first and the fourth coefficient are zero is impossible. Hence, the only remaining case is that the second and third coefficient are zero which means $\frac{a_u}{b_u} = \frac{\overline{a}_{w+k}}{\overline{b}_{w+k}}$ and $\frac{a_{u+k}}{b_{u+k}} = \frac{\overline{a}_{w}}{\overline{b}_{w}}$. From comparing the left-hand side and the right-hand side of \cref{eq:Part2_Equiv2}, it follows that $u=w$. Next, we use \cref{eq:Part2_Equiv1}. By the same argument as in case~\ref{item:b}, it can be shown that this equation never holds. Hence, case \ref{item:a} is impossible.\par 
		
		In summary, the only possible choice for $L_A(X,Y)$ and $L_B(X,Y)$ is that both polynomials are monomials of the same degree. Hence, we have either 
		\begin{align*}
			L_A(X,Y) = L_1(X)&& \text{and}&& L_B(X,Y) = L_4(Y)
		\end{align*}
		or
		\begin{align*}
			L_A(X,Y) = L_3(Y)&& \text{and} &&L_B(X,Y) = L_2(X).
		\end{align*}
		 We will show that in both cases, $s=t$. Consider the first case. Let
		\begin{align*}
			L_A(X,Y) = a_u X^{2^u} &&\text{and}	&&L_B(X,Y) = \overline{b}_u Y^{2^u}.
		\end{align*}
		 Then \cref{eq:Part2_Equiv1} becomes
		\begin{align}
		\label{eq:final_case1}
			(a_u x^{2^u})^{2^k+1} + \alpha (\overline{b}_u y^{2^u})^{(2^k+1)2^s} = N_1(x^{2^k+1} + \alpha y^{(2^k+1)2^t}) + N_3(xy) + M_A(x,y).
		\end{align}
		Consequently, $M_A(X,Y) = 0$ and $N_3(X) = 0$. Moreover, $N_1(X)$ has to be a monomial and $s = t$. Next, we consider the second case. Let 
		\begin{align*}
			L_A(X,Y) = \overline{a}_u Y^{2^u} &&\text{and} &&L_B(X,Y) = b_u X^{2^u}. 
		\end{align*}
		Now, \cref{eq:Part2_Equiv1} is
		\begin{align}
		\label{eq:final_case2}
			(\overline{a}_u y^{2^u})^{2^k+1} + \alpha (b_u x^{2^u})^{(2^k+1)2^s} = N_1(x^{2^k+1} + \alpha y^{(2^k+1)2^t}) + N_3(xy) + M_A(x,y).
		\end{align}
		It follows that $M_A(X,Y) = 0$ and $N_3(X) = 0$. Moreover, $N_1(X)$ has to be a monomial. Assume $N_1(X) = c_r X^{2^r}$. Then \cref{eq:final_case2} becomes
		\[
			\overline{a}_u^{2^k+1} y^{2^u(2^k+1)} + \alpha b_u^{2^s(2^k+1)} x^{2^{u+s}(2^k+1)} = c_r x^{2^r(2^k+1)} + \alpha^{2^r} c_r y^{2^{r+t}(2^k+1)}.
		\]
		Consequently, we need $u \equiv r+t \pmod{m}$ and $u+s \equiv r \pmod{m}$ which is equivalent to $s \equiv -t \pmod{m}$. As $0 \le s,t \le \frac{m}{2}$, this equation only holds for $s = t = 0$ and $s = t = \frac{m}{2}$. Hence, in this second case, we also obtain $s=t$. This concludes our proof.
	\end{proof}

	From \autoref{th:ZP_inequivalence1}, we immediately obtain \autoref{cor:numberofAPN1} which gives a lower bound on the total number of CCZ-inequivalent APN functions on $\Fptwom$, where $m$ is even. We recall \autoref{cor:numberofAPN1} and give a short proof.
	
	\begin{cor:ZP}
		On $\Fptwom$, where $m \ge 4$ is even, there exist 
		\[
		\left(\left\lfloor\frac{m}{4}\right\rfloor+1\right) \frac{\varphi(m)}{2}
		\]
		CCZ-inequivalent Pott-Zhou APN functions from \autoref{th:ZhouPottAPN}, where $\varphi$ denotes Euler's totient function.
	\end{cor:ZP}
	\begin{proof}
		According to \autoref{th:ZP_inequivalence1}, for $0 < k, \ell < \frac{m}{2}$ and $0 \le s,t \le \frac{m}{2}$, two Pott-Zhou APN functions $f_{k,s}$ and $f_{\ell,t}$ on $\Fptwom$, where $m$ is even, are CCZ-inequivalent if and only if $(k,s) \ne (\ell,t)$. We count the number of distinct pairs $(k,s)$ that can be chosen: as $0 \le s \le \frac{m}{2}$ and $s$ is even, we have $\lfloor\frac{m}{4}\rfloor$ nonzero choices for $s$ plus the choice $s=0$. As $0 < k < \frac{m}{2}$ and $\gcd(k,m)=1$, we have $\frac{\varphi(m)}{2}$ choices for $k$. 
	\end{proof}
	In \autoref{tab:ZP-APN_smallm}, we present the result of \autoref{cor:numberofAPN1} for small values of $m$. Note that from computational results, only the number of inequivalent Pott-Zhou APN functions for $m=2$ and $m=4$ was known.\par
	
	\begin{table}[h]
		\centering
		\caption{Number of CCZ-inequivalent Pott-Zhou APN functions on $\Fptwom$ for small values of $m$.}
		\label{tab:ZP-APN_smallm}		
		\begin{tabular}{*{18}{r}}
			\hline
			\rowcolor{gray!30}m &2 &4 &6 &8 &10 &12 &14 &16 &18 &20 &22 &24 &26 &28 &30 &32 &34\\
			\# 	&1 &2 &2 &6 &6 &8 &12 &20 &15 &24 &30 &28 &42 &48 &32 &72 &72\rule[.5em]{0em}{.5em}\\\hline
		\end{tabular}		
	\end{table}

	In \autoref{fig:ZP-APN}, the result is illustrated for $m \le 1000$. Note that the upper bound on the number of Pott-Zhou APN functions of $m(m+4)/16$ holds for all $m \ge 4$. It is sharp whenever $m$ is a power of $2$. The lower bound of $m\sqrt{m}/2$ holds for $m > 210$. 
	
	\begin{figure}[h]
		\centering
		\begin{tikzpicture}[every mark/.append style={mark size=.5pt, black}]
		\begin{axis}[
		xlabel={m},
		grid=both,
		minor tick num = 1,
		width=14cm,
		height=10cm,
		xmin=-50,
		ymin=-5000,
		xmax=1050,
		ymax=65000,
		scaled y ticks = false,
		]
		\addplot table [y=num,col sep=comma, only marks] {graph3.csv};
		\addplot[gray,thick,samples=2000, domain=0:1000,] {x*(x+4)/16} node [pos=.95, rectangle, left = .3cm, fill = white]{$\frac{m(m+4)}{16}$};
		\addplot[gray,thick,samples=2000, domain=0:1000] {x*sqrt(x)/2} node [pos=.95, rectangle, below = .2cm, fill = white]{$\frac{m\sqrt{m}}{2}$};		
		\end{axis}
		\end{tikzpicture}
		\caption{Number of CCZ-inequivalent Pott-Zhou APN functions on $\Fptwom$ for $m \le 1000$.}
		\label{fig:ZP-APN}
	\end{figure}
	
	Moreover, from the proof of \autoref{th:ZP_inequivalence1}, we can deduce the order of the automorphism groups of the Pott-Zhou APN functions.
	
	\begin{theorem}
	\label{cor:ZP-automorphismgroup}
		Let $f_{k,s}$ be an APN function from \autoref{th:ZhouPottAPN} on $\Fptwom$, where $m$ is even. If $m \ge 4$, then
		\[
			|\Aut_L(f_{k,s})| = 
				\begin{cases}
					3m(2^m-1)			&\text{if } s \in \{0,\frac{m}{2}\},\\
					\frac{3}{2}m(2^m-1)	&\text{otherwise,}
				\end{cases} 
		\]
		and
		\[
			|\Aut(f_{k,s})| =
				\begin{cases}
					3m2^{2m}(2^m-1) 		&\text{if } s \in \{0,\frac{m}{2}\},\\
					3m2^{2m-1}(2^m-1)	&\text{otherwise.}
				\end{cases} 			
		\]
		If $m=2$, then
		\begin{align*}
			|\Aut_L(f_{1,0})| &= 360& \text{and} &&|\Aut(f_{1,0})| &= 5760.
		\end{align*}
	\end{theorem}
	\begin{proof}
		For $m \ge 6$, consider \cref{eq:final_case1} and assume $N_1(X) = c_u X^{2^u}$. Then \cref{eq:final_case1} becomes
		\[
			a_u^{2^k+1} x^{2^u(2^k+1)} + \alpha \overline{b}_u^{2^s(2^k+1)} y^{2^{s+u}(2^k+1)} = c_u x^{2^u(2^k+1)} + \alpha^{2^u}c_uy^{2^{s+u}(2^k+1)}.
		\]
		It follows that
		\begin{align*}
			a_u^{2^k+1} &= c_u &\text{and}	&& \alpha \overline{b}_u^{2^s(2^k+1)} &= \alpha^{2^u} c_u
		\end{align*}
		which is equivalent to
		\[
			\alpha^{2^u-1} a_u^{2^k+1} = \overline{b}_u^{2^s(2^k+1)}.
		\]
		This equation can only hold if $\alpha^{2^u-1}$ is a cube, which is the case if and only if $u$ is even. Hence, we first have $\frac{m}{2}$ choices for $u$. Then we can choose $a_u$ from $\Fpm^*$, hence we have $2^m-1$ choices for $a_u$. Finally, every choice of $a_u$ results in $3$ choices for $\overline{b}_u$, since $x \mapsto x^{2^k+1}$ is a $3$-to-$1$ mapping on $\Fpm^*$. If $s \in \{0, \frac{m}{2}\}$, we additionally obtain the same amount of choices as from \cref{eq:final_case1} also from \cref{eq:final_case2}. Consequently, we have twice as many total possibilities in this case.\par
		
		As \autoref{th:ZP_inequivalence1} only holds for $m \ge 6$, we have checked the cases $m = 2$ and $m=4$ computationally with \texttt{Magma}~\cite{magma}. More precisely, we computed the automorphism groups of the related codes
		\[
			\begin{pmatrix}1\\x\\f_{k,s}(x)\end{pmatrix},
		\]
		where $x \in \Fptwom$. For $m = 4$, we obtain the same result as for $m \ge 6$. However, for $m=2$, the only Pott-Zhou APN function on $\F_{2^4}$ is linearly equivalent to the Gold APN~function $x \mapsto x^3$. Hence, $|\Aut_L(f_{1,0})| = 360$ and $|\Aut(f_{1,0})| = 5760$ as we have shown in \autoref{cor:GoldAPN_AutomorphismGroup}.
	\end{proof}

\section{Conclusion and open questions}
	In the present paper, we establish a lower bound on the total number of CCZ-inequivalent APN functions on the finite field $\Fptwom$, where $m$ is even. We show that in any such field plenty of these functions do exist. From this result, the following questions arise naturally:
	\begin{itemize}
		\item Is there a similar lower bound on the number of APN functions on $\F_{2^n}$, where $n$ is not a multiple of 4? To answer this question, a closer look at the list by \textcite[Table~3]{budaghyan2019} might be helpful: first, to check on which fields the respective functions do exist, and second, to see whether the construction depends on parameters that could provide inequivalent APN functions.
		\item Can the lower bound presented in this paper be improved? From the list by \textcite{budaghyan2019}, the class (F11), that was recently discovered by \textcite{taniguchi2019}, seems to be a canonical starting point to work on this problem as its structure is very similar to the structure of the Pott-Zhou APN functions.
	\end{itemize}
	
\section*{Acknowledgments}
	The authors thank Satoshi Yoshiara for his valuable comments about the automorphism groups of quadratic APN functions under EA- and CCZ-equivalence.

\printbibliography
	
\end{document}